\documentclass[letterpaper,10pt,reqno]{amsart}
\usepackage[usenames,dvipsnames]{xcolor}

\usepackage[portrait,margin=3.5cm]{geometry}
\usepackage{mathrsfs}

\usepackage[foot]{amsaddr}
\usepackage{amssymb,amsthm,amsmath,amsfonts,amsbsy,latexsym,dsfont}
\usepackage[english]{babel}

\usepackage{color}
\usepackage{graphicx}
\usepackage{bbm}
\usepackage{comment}
\usepackage{microtype}
\usepackage{todonotes}
\definecolor{MyDarkblue}{rgb}{0,0.08,0.50}
\definecolor{Brickred}{rgb}{0.65,0.08,0}
\usepackage[colorlinks,citecolor=blue,urlcolor=blue,linkcolor=blue]{hyperref}

\hypersetup{
colorlinks=true,       
    linkcolor=MyDarkblue,          
    citecolor=Brickred,        
    filecolor=red,      
    urlcolor=cyan           
}

\newtheorem*{theorem*}{Theorem}
\newtheorem{theorem}{Theorem}[section]
\newtheorem{lemma}[theorem]{Lemma}
\newtheorem{example}[theorem]{Example}
\newtheorem{proposition}[theorem]{Proposition}
\newtheorem{corollary}[theorem]{Corollary}

\newtheorem{definition}[theorem]{Definition}
\newtheorem{assumption}[theorem]{Assumption}
\newtheorem{remark}[theorem]{Remark}
\newtheorem{claim}[theorem]{Claim}

\newtheorem{observation}[theorem]{Observation}

\newcommand{\Pv}{\mathbb{P}}

\newcommand{\Ev}{\mathbb{E}}

\newcommand{\CC}{\mathcal{C}}

\newcommand{\sss}{\scriptscriptstyle}
\newcommand{\Var}{{\rm Var}}

\newcommand{\CE}{{\mathcal{E}}}

\newcommand*{\CMD}{{\mathrm{CM}}_n(\boldsymbol{d})}

\newcommand{\e}{{\mathrm e}}

\numberwithin{equation}{section}

\newcommand{\R}{\mathbb{R}}
\newcommand{\N}{\mathbb{N}}
\newcommand{\Z}{\mathbb{Z}}

\renewcommand{\emptyset}{\varnothing}

\newcommand{\CA}{\mathcal {A}}
\newcommand{\CB}{\mathcal {B}}
\newcommand{\CD}{\mathcal {D}}
\newcommand{\CF}{\mathcal {F}}
\newcommand{\CG}{\mathcal {G}}

\newcommand{\CL}{\mathcal {L}}

\newcommand{\CN}{\mathcal {N}}

\newcommand{\CP}{\mathcal {P}}

\newcommand{\CS}{\mathcal {S}}
\newcommand{\CT}{\mathcal {T}}

\newcommand{\CH}{\mathcal {H}}

\newcommand*{\vr}{\varrho}

\newcommand*{\ve}{\varepsilon}

\newcommand*{\be}{\begin{equation}}
\newcommand*{\ee}{\end{equation}}
\newcommand*{\ba}{\begin{aligned}}
\newcommand*{\ea}{\end{aligned}}
\newcommand*{\barr}{\begin{array}{c}}
\newcommand*{\earr}{\end{array}}
\def \toinp    {\buildrel {\Pv}\over{\longrightarrow}}
\def \toindis  {\buildrel {d}\over{\longrightarrow}}
\def \toas     {\buildrel {a.s.}\over{\longrightarrow}}

\newcommand*{\Yrn}{Y_r^{\scriptscriptstyle{(n)}}}
\newcommand*{\Ybn}{Y_b^{\scriptscriptstyle{(n)}}}

\newcommand*{\bnr}{b_n^{\sss{(r)}}}
\newcommand*{\bnb}{b_n^{\sss{(b)}}}
\newcommand*{\bnq}{b_n^{\sss{(q)}}}

\newcommand*{\wit}{\widetilde}

\newcommand*{\uiq}{u_{i}^{\sss{(q)}}}
\newcommand*{\uiqp}{u_{i+1}^{\sss{(q)}}}
\newcommand*{\huiq}{\widehat u_{i}^{\sss{(q)}}}
\newcommand*{\ind}{\mathbbm{1}}
\def\namedlabel#1#2{\begingroup
    #2%
    \def\@currentlabel{#2}%
    \phantomsection\label{#1}\endgroup
}

\begin{document}
	\title[When is a scale-free graph ultra-small?]{
When is a scale-free graph ultra-small?}

	\date{\today}

	\author[van der Hofstad]{Remco van der Hofstad}
	\address{Department of Mathematics and
	    Computer Science, Eindhoven University of Technology, P.O.\ Box 513,
	    5600 MB Eindhoven, The Netherlands.}
\author[Komj\'athy]{J\'ulia Komj\'athy}
	\email{rhofstad@win.tue.nl, j.komjathy@tue.nl}

\begin{abstract}
In this paper we study typical distances in the configuration model, when the degrees have asymptotically infinite variance. We assume that the empirical degree distribution follows a power law with exponent $\tau\in (2,3)$, up to value $n^{{\beta_n}}$ for some ${\beta_n}\gg (\log n)^{-\gamma}$ and $\gamma\in(0,1)$.
This assumption is satisfied for power law i.i.d.\ degrees, and also includes \emph{truncated power-law} empirical degree distributions where the (possibly exponential) truncation happens at $n^{{\beta_n}}$.  These examples are commonly observed  in many real-life networks.

 We show that the graph distance between two uniformly chosen vertices centers around $2 \log \log (n^{{\beta_n}}) / |\log (\tau-2)| + 1/({\beta_n}(3-\tau))$, with tight fluctuations. Thus, the graph is an \emph{ultrasmall world} whenever $1/{\beta_n}=o(\log\log n)$.
 We determine the distribution of the fluctuations around this value, in particular we prove these form a sequence of tight random variables with distributions that show $\log \log$-periodicity, and as a result it is non-converging.

 We describe the \emph{topology and number of shortest paths}: We show that the number of shortest paths is of order $n^{f_n{\beta_n}}$, where $f_n \in (0,1)$ is a random variable that oscillates with $n$.  We decompose shortest paths into three segments, two `end-segments' starting at each of the two uniformly chosen vertices, and a middle segment. The two end-segments of any shortest path have length $\log \log (n^{{\beta_n}}) / |\log (\tau-2)|$+tight,  and the total degree is increasing towards the middle of the path on these segments.  The connecting middle segment has length $1/({\beta_n}(3-\tau))$+tight, and it contains only vertices with degree at least of order $n^{(1-f_n){\beta_n}}$, thus all the degrees on this segment are comparable to the maximal degree.

 Our theorems also apply when instead of truncating the degrees, we start with a configuration model and we \emph{remove} every vertex with degree at least $n^{{\beta_n}}$, and the edges attached to these vertices. This sheds light on the \emph{attack vulnerability} of the configuration model with infinite variance degrees.
 \end{abstract}

\maketitle

\section{Introduction and results}
Many real-world networks are claimed to be {\em small worlds}, meaning that their graph distances are quite small. In social networks, such small distances go under the name of the `six-degrees-of-separation' paradigm and have attracted attention due to the interesting experiments by Milgram \cite{Milg67, TraMil69}. See also Pool and Kochen \cite{PooKoc78} as well as \cite{DodMuhWat03}, where a related  experiment is described on the basis on email messages. After this start in social sciences, the small-world nature of many other networks was first described by Strogatz and Watts \cite{WatStr98}. A popular account of small-world aspects of networks can be found in the book by Watts \cite{Watt99}. See also the surveys by Newman \cite{Newm03a} and Albert and Barab\'asi \cite{AlbBar02} on real-world networks.

The reported small-world nature of many real-world networks has incited a deep and thorough study of typical distances in random graphs. See the highly influential paper by Newman, Strogatz and Watts \cite{NewStrWat00}, who pioneered this line of research. There is a deep relation between the small-world nature of networks and their other often reported common feature, the {\em scale-free paradigm}, which states that the proportion of vertices of degree $k$ in many networks scales as an inverse power of $k$ for $k$ large. This scale-free nature implies that there are many vertices with very high degrees, and these {\em hubs} drastically shrink graph distances. The common picture in many random graph models is that graph distances are asymptotically \emph{logarithmic} in the graph size when the degrees have finite-variance degrees \cite{FerRam04, HofHooVan05a, NewStrWat00}, while they are \emph{doubly logarithmic} or {\em ultrasmall} when the graphs have infinite-variance degrees \cite{ChuLu01, ChuLu04, CohHav03, DomHofHoo10, HHZ07, NorRei06}. See also \cite{H102} for a discussion of many of the available results. The conclusion is that typical distances in random graphs are closely related to their degree structure, and larger degrees significantly shrink graph distances.

In many real-world networks, not only power-law degree sequences are observed, but also power laws with a so-called {\em exponential truncation}. This means that, even though the degree distribution for small values is close to a power law, for large values the tail distribution becomes exponentially small. This occurs e.g.\ in sexual networks \cite{LilEdlAmaSta01}, the Internet Movie Data base \cite{AmaScaBarSta00}, and in scientific collaboration networks \cite{Newm01b, Newman2001}.
Of course, it is not easy to guess whether a distribution has a power law, or rather a power law with exponential truncation. Newman and collaborators give sensible suggestions on how to approach these issues in real-world networks \cite{ClaShaNew09, Newm05}.

The value above which the exponential decay starts to set in is the {\em truncation parameter}. Naturally, when the degrees already had finite variance before truncation, they will remain to have so after truncation, so nothing much happens and graph distances ought to behave as in the finite-variance setting in \cite{HofHooVan05a}. Further, any power-law distribution with exponential truncation with a {\em bounded} truncation parameter has all moments, and thus distances ought to become \emph{logarithmic} as in the finite-variance case in e.g., \cite{HofHooVan05a}, accounting to `strictly' \emph{small-world networks} \cite{WatStr98}. The situation changes dramatically when dealing with infinite-variance degrees, that is, when the power-law exponent is below $3$. Indeed, in this setting, it is well known that typical paths realizing the graph distance pass through the vertices of highest possible degrees (i.e., the {\em hubs}) and typical distances grow much slower, \emph{doubly-logarithmically} with the size of the graph, accounting to \emph{ultrasmall worlds} \cite{HHZ07,NorRei08}. 
However, truncating the degrees could possibly have a dramatic effect, and could possibly increase the graph distances rather substantially.
In fact, even though the asymptotic variance of the degrees remains infinite after truncation, distances might grow substantially with the truncation and the graph might fail to be an ultra-small world.

The main aim of this paper is to quantify the effect of truncation of the degrees in random graphs models, in particular, in the configuration model that we define below. For simplicity, we state our result as a theorem without bothering to explain the somewhat tedious details and conditions. Below, we give several more accurate and detailed versions of this theorem.
\begin{theorem*}\label{thm:meta}
Let us consider the configuration model on $n$ vertices with empirical degree distribution that follows a power law with exponent $\tau\in(2,3)$, satisfying some appropriate regularity assumptions and truncated at degrees $n^{{\beta_n}}$, where ${\beta_n} (\log n)^\gamma\to \infty$ for some $\gamma\in(0,1)$. Let $\mathrm d_G(v_r, v_b)$ denote the graph distance between two uniformly chosen vertices $v_r, v_b$. Then
\be\label{eq:meta-dist} d_G(v_r, v_b) - 2\frac{\log \log (n^{{\beta_n}})}{|\log (\tau-2)|} -\frac{1}{(3-\tau) {\beta_n}} \ee
is a tight sequence of random variables.
\end{theorem*}
For the precise result see Theorem \ref{thm:distances} below. In that theorem, we also identify the distributional limit of the tight sequence of random variables along subsequences. This was first done in \cite{HHZ07} using different methods for the special case $\beta_n\equiv 1/(\tau-1)$. We show in Section \ref{s:compare} that the seemingly different formulations, Theorem \ref{thm:distances} and the result in \cite{HHZ07} are actually the same.

Note that as soon as ${\beta_n}=o(1/\log\log n)$,  the term containing $1/{\beta_n}$ becomes dominant, and the second order term is of order $\log\log n$. Thus, the random graph fails to be an ultrasmall world in a strict sense as soon as the truncation happens at degrees at most $n^{o(1/\log\log n)}$. However, when ${\beta_n} \log \log n \to \infty$, the dominant term is the first one, of order $\log \log n$.
This theorem shows that distances in truncated power-law graphs with infinite asymptotic variance \emph{interpolate} between ultrasmall worlds and small worlds.

This sheds light to a discussion in the physics literature \cite{CohHav03, Doro2003, Fron2003, FroFro04, FroFro05, NewStrWat00} about the validity of the formula derived using generating function methods or $n$-dependent branching process approximations, stating that
\be\label{eq:physics} \mathrm d_G(v_r, v_b) = \frac{\log n}{\log \nu_n} + \text{tight} = \frac{1}{(3-\tau){\beta_n}}+\text{tight}\ee
with $\nu_n$ standing for the empirical second moment\footnote{We show in Claim \ref{cl:tech} and \eqref{eq:edn} below that $\nu_n\sim n^{{\beta_n}(3-\tau)}$, thus, $\log n/\log \nu_n$ in \eqref{eq:physics} yields $1/({\beta_n}(3-\tau))$, justifying the equality between the two expressions.} of the degrees and ${\beta_n}$ is the truncation exponent.
Compare this formula to the one in \eqref{eq:meta-dist} and note that this formula yields the first order term if and only if ${\beta_n}=o(1/\log \log n)$, while even in this case it fails to capture the second order term, which is of order $\log \log n$.

While \cite{NewStrWat00} questions the validity of this formula for $\tau\in(2,3)$, \cite{Fron2003, FroFro04, FroFro05}  argue that a constant ${\beta_n}\equiv {\beta}$ for the truncation exponent yields bounded typical distances in this regime, in agreement with what \eqref{eq:physics} suggests. This contradicts the arguments in \cite{CohHav03} where the authors show that the smallest achievable order for typical distances is $\log\log n$.

Closest to our result is the work of Dorogovtsev \emph{et al.} \cite{Doro2003}, who, using (non-rigorous) generating function methods, study typical distances in the configuration model with truncated power-law distributions of the form $\Pv(D=k)\sim k^{-\tau} \zeta^k$, for constant $\zeta<1$, and derive that
\be \ba \label{eq:doro} \mathrm d_G(v_r, v_b) &= \frac{\log n}{(3-\tau)|\log (1-\zeta)|} +  \frac{2\log |\log (1-\zeta)|}{|\log (\tau-2)|}+\text{tight},\\
&=\frac{1}{(3-\tau){\beta_n}} +  \frac{2 \log \log (n^{{\beta_n}})}{|\log (\tau-2)|} + \text{tight},
\ea\ee
where we took the liberty to set $\zeta:=\exp\{-1/n^{{\beta_n}}\}$ to obtain the second line. Observe that our (rigorous) result in \eqref{eq:meta-dist} above is in perfect agreement with this.  The region of validity of their approach \cite[(106)]{Doro2003} translates to the requirement that ${\beta_n}=o(1/\log \log n)$, can be interpreted that their method requires that the first term is the leading order term in \eqref{eq:doro}.
We explain below in Section \ref{s:discussion} why the term $\log \log n$ is missing from \eqref{eq:physics} and where it comes from.

Let us make a comparison: two models with the same maximal degree, a truncated and an un-truncated one. More precisely, take a model with power-law exponent $\tau\in (2,3)$ and truncation at ${\beta_n}={\beta}<1/(\tau-1)$ fixed. In a model with power-law exponent $\wit\tau$ and natural truncation (that is, $\wit {\beta}=1/(\wit \tau -1)$), the maximal degree is  $n^{1/(\wit\tau-1)}$. Setting the maximal degrees in the two models, -- $n^{{\beta}}$ and $n^{1/(\wit\tau-1)}$, respectively -- to be equal yields the relation $\wit\tau=1+1/{\beta}$. Comparing distances in the two models, we see an interesting phenomenon. When ${\beta_n}<1/2$, the un-truncated model has $\wit\tau=1+1/{\beta}>3$, implying that  $\lim_{n\to\infty}\Ev[(D_n')^2 ]<\infty$, which, in turn, implies that  distances jump up to \emph{logarithmic} order. In the truncated model distances are of order $\log\log n$, as described by \eqref{eq:meta-dist}. When ${\beta_n}>1/2$, the leading order of distances in the truncated model is $\log \log n/|\log (\tau-2)|$+tight, while in the un-truncated model it is $\log \log n/|\log (\wit\tau-2)|+tight=\log \log n/|\log (1/{\beta}-1)|+tight$. Since ${\beta}<1/(\tau-1)$ by assumption, $1/{\beta}-1>\tau-2$ and hence the distances in the latter model are larger. This means that `re-parametrizing' the truncated model by another power-law ($\tau'$) that would more naturally reflect the maximal degree is not the same as truncating, even the leading term changes. This effect is extreme when ${\beta}<1/2$, in which case the truncated and the un-truncated models do not even belong to the same universality class (ultrasmall versus small world). 

Let us comment on the choice of model as well. We expect that the same result is true for the giant component of the Chung-Lu or Norros-Reitu model when the power law exponent is $\tau\in(2,3)$. These models behave qualitatively similarly to the configuration model. In fact, the ultrasmall nature of these networks were pointed out in \cite{NorRei06, NorRei08}. The independence of the edges conditioned on vertex-weights even makes path-counting methods easier, in fact, we expect that the same proof that we provide here could be applied for these models as well. 

\emph{Attack vulnerability.}
The removal of all vertices above a certain degree in a network is called a \emph{targeted attack} or \emph{deliberate attack}.
An immediate corollary (see Corollary \ref{cor:attack} below) of our work is that our results remain valid when instead of truncating the degrees, we \emph{remove} all vertices with degree at least $n^{\wit{\beta_n}}$ from a configuration model. In particular, distances after a targeted attack are described by \eqref{eq:meta-dist}, with $\beta_n$ replaced by $\wit\beta_n$. This theorem also sheds light to how distances gradually grow from ultrasmall to small world when vertices with  smaller and smaller degree are gradually  removed. A similar analysis have been carried out for a variant of the preferential attachment model in \cite{EckMor14}, for the special case when all degrees above order $\log n$ (equivalently, the oldest $\ve n$ vertices for small small $\ve>0$) are removed.   They show that distances in this case grow logarithmically, giving a strong base for our conjecture that the formula in \eqref{eq:meta-dist}  should also be valid for $\beta_n=\Theta(1/\log n)$, which is, at least with the methods of this paper, beyond our reach.

\subsection{The model and the main result}
In this paper we work under the setting of  the configuration model $\CMD$. In this random graph model, there are $n$ vertices, with prescribed degrees $d_v, v\in \{1,2,\dots, n\}:=[n]$. To each vertex $v\in [n]$ we assign $d_v$ half-edges and the half-edges are then paired uniformly at random to form
edges.
We assume that the total number of half-edges $\ell_n:=\sum_{v\in[n]} d_v$ is even. We denote the outcome - a graph-valued random variable - by $\CMD$.
\subsubsection{Setting and assumptions}
We study the case when the empirical degree distribution follows a possibly truncated power law, with an exponent that gives rise to empirical variance tending to infinity with $n$ and when the truncation happens at some polynomial of $n$. To make this precise, we impose the following three assumptions on the \emph{empirical degree distribution}, $F_n(x):=\frac1n \sum_{v\in [n]} \ind_{\{d_v\le x\}}$:

\begin{assumption}[Power-law tail behavior]\label{assu:degree-dist}
	There exists a  ${\beta_n}\in (0, 1/(\tau-1)]$ such that  for all $\ve>0$,
$F_n(x)=1$ for $x\ge n^{{\beta_n}(1+\ve)}$, while for
 all $x\le n^{{\beta_n}(1-\ve)}$,
\be\label{eq:F} 1- F_n(x)= \frac{L_n(x)}{x^{\tau-1}},\ee
 with $\tau \in (2,3)$, and a function $L_n(x)$ that satisfies for some constant $C_1>0$ and $\eta\in (0,1)$ that
 \be\label{eq:L}\exp\{ -C_1 (\log x)^\eta  \} \le L_n(x)\le \exp\{ C_1 (\log x)^\eta  \}.\ee
\end{assumption}
\begin{assumption}[Minimal degree at least $2$]\label{assu:mindeg}
$\min_{v\in [n]} d_v \ge 2$.
\end{assumption}
See Remark \ref{rem:no-mindeg} for an extension of our results in the case when $\min_{v\in [n]} d_v =1$.
We write $D_n$ for a random variable having distribution $F_n$. Then $D_n$ is the degree of a uniformly chosen vertex from $[n]$. We introduce $D^\star_n:=$(the \emph{size-biased version} of $D_n)-1$ by
 	\be\label{def:size-biased1}
	\Pv(D^\star_n=j):=\frac{j+1}{\ell_n} \sum_{v\in[n]} \ind_{\{d_v=j+1\}}=\frac{(j+1)\Pv(D_n=j+1)}{\Ev[D_n]}\text{, }\quad j\geq 0.
	\ee
We write $F^\star_n(x)$ for the distribution function of $D^\star_n$.
Note that for all  $x\le n^{{\beta_n}(1-\ve)}$, \[ 1- F_n^\star(x)=\frac{1}{\Ev[D_n]}\sum_{j\ge x} (j+1)\Pv(D_n=j+1)\ge\frac{1}{\Ev[D_n]} x (1-F_n(x)), \]
and  similarly, under Assumption \ref{assu:degree-dist},
\[ \ba 1- F_n^\star(x)&\le (x+1) (1-F_n(x+1)) + \sum_{j\ge x} (1-F_n(x)) \\
&\le \frac{2}{\Ev[D_n]} \frac{L_n(x+1)}{x^{\tau-2}} + \frac{1}{\Ev[D_n]}\sum_{j\ge x} \frac{L_n(j)}{j^{\tau-1}}.\ea\]
By a Karamata-type theorem (see \cite[Proposition 1.5.10]{BinGol}) the latter sum on the rhs is at most $2 L_n(x)/(x^{\tau-2} (\tau-2))$ for all large enough $x$. Further, for $x\ge n^{{\beta_n}(1+\ve)}$ it is obvious from the definition \eqref{def:size-biased1} that $1-F_n^\star(x)=0$.
Thus, it follows
that for all $\ve>0$ and $x\le n^{{\beta_n}(1-\ve)}$, $F_n^{\star}$ satisfies that
\be \label{eq:size-biased2}  1- F^\star_n(x)= \frac{L^\star_n(x)}{x^{\tau-2}},\ee
with a function $L^\star_n(x)$ satisfying \eqref{eq:L} again (possibly with a different constant $C^\star_1$ instead of $C_1$ in the exponent in \eqref{eq:L}  for $L_n^\star$).

To be able to state convergence results, we will need an assumption that relates the behavior of $F_n$ and $F_n^\star$ for different values of $n$ to a limiting distribution function. 
We write $\mathrm{d}_{\sss{\mathrm{TV}}}(F, G):=\tfrac12 \sum_{x\in \N} |F(x+1)-F(x) - (G(x+1)-G(x)) |$ for the total variation distance between two (discrete) probability measures.
The weakest form of such assumption that we can pose is captured in the following assumption:
\begin{assumption}[Convergence to limiting distributions]\label{assu:tv}
We assume that there exist distribution functions $F(x)$, $F^\star(x)$ such that $F_n(x)\to F(x)$ and $F_n^\star(x)\to F^\star(x)$ in all continuity points of $F(x)$, $F^\star(x)$. We assume that there exists a $\kappa>0$,
\be\label{eq:tv-dist-1} \max\{ \mathrm{d}_{\sss{\mathrm{TV}}}(F_n, F), \mathrm{d}_{\sss{\mathrm{TV}}}(F_n^\star, F^\star) \} \le n^{-{\beta_n}\kappa}.\ee
\end{assumption}
 Let us write $D, D^\star$ for random variables following the limiting distributions $F, F^\star$ in Assumption \ref{assu:tv}, respectively.  Since total variation convergence equals weak convergence for discrete random variables, $D_n\toindis D$ and $D_n^\star\toindis D^\star$. Using \eqref{def:size-biased1}, we further obtain that
 \be \label{eq:limit-sb} \Pv(D^\star=j) = \frac{(j+1)\Pv(D=j)}{\Ev[D]}, \ee
thus $D^\star$ is the (size-biased version$-1$) of $D$.
 Further note that $F$ can be written in the form \eqref{eq:F}, and $F^\star$ as in \eqref{eq:size-biased2} with limiting $L, L^\star$ that satisfies \eqref{eq:L}. Note that the limit variables are not truncated. Further, the bound $n^{-{\beta_n} \kappa}$ in Assumption \ref{assu:tv} is best possible since $\mathrm{d}_{\sss{\mathrm{TV}}}(F_n, F) \ge \Pv(D>n^{{\beta_n}}) \ge n^{-{\beta_n}(\tau-1)}\ell(n^{-{\beta_n}(\tau-1)})\ge n^{-{\beta_n}(\tau-1-\delta)}$. It is also reasonable, e.g. it can be shown that it is satisfied in Examples \ref{ex:iid}-\ref{ex:trunc-hard} below. 

Under Assumptions \ref{assu:degree-dist} and \ref{assu:mindeg}, the graph almost surely has a unique connected component of size $n(1-o_{\Pv}(1))$ see e.g.\ \cite[Vol II., Theorem 4.1]{H10} or \cite{MolRee95, MolRee98}, or the recent paper \cite{FedHof16}.

We provide three examples in Section \ref{s:examples} below (i.i.d.\ degrees, exponential and hard truncation) that satisfy Assumption \ref{assu:degree-dist}, see Examples \ref{ex:iid}, \ref{ex:trunc}, \ref{ex:trunc-hard} below, as well as collect some references to networks following such empirical degree distributions.

In this paper we study typical distances, that is, the graph distance $\mathrm d_G$ between \emph{two uniformly chosen vertices} in the graph. For the sake of the proof, we denote these vertices by $v_r$ and $v_b$ and think of them as being red and blue, respectively.
\begin{definition}[With high probability]
We say that a sequence of events $\CE_n$ happens with high probability under the measure $\mathbb Q$ (and abbreviate this as $\mathbb Q$-whp), if $\mathbb Q(\CE_n) \to 1$
as $n\to \infty$. We write simply `whp' when the measure is the annealed measure of the configuration model and the two uniformly chosen vertices $v_r, v_b$.
\end{definition}
We emphasize that in the setting of this paper, whp statements should be read as follows: for asymptotically almost every realizations of the random graph $\CMD$ and almost all pairs of vertices $(v_r, v_b)$, the statement is true.
\begin{definition}[$\sim$ notation]\label{def:sim}
We use the shorthand notation $X_n \sim a_n$ if there exists a constant $\theta\in(0,1)$ such that
\be\label{eq:sim} X_n \sim a_n \quad \Longleftrightarrow\quad \Pv\left( X_n \in  [a_n \e^{-(\log a_n)^{\theta}}, a_n \e^{(\log a_n)^{\theta}}] \right) \to 1.\ee We call vertices with degree at least $\sim n^{(\tau-2){\beta_n}}$ \emph{hubs}.
\end{definition}
Note that $X_n\sim n^a$ is somewhat stronger than stating that $X_n=n^{a(1+o_{\Pv}(1))}$.

The statement of the main theorem uses some knowledge about infinite-mean branching processes, as well as
their coupling to the local neighborhood of the two vertices $v_r, v_b$. So, before stating the result, we have to do a small excursion into defining these objects.
In particular, Lemma \ref{lem:couple} and Corollary \ref{corr:couple} below, based on \cite[Proposition 4.7]{BHH10}, states that under Assumption \ref{assu:tv}, whp, the number of vertices and their forward degrees in an exploration of the neighborhood of $v_r, v_b$ can be coupled to two independent branching processes (that are embedded in the graph disjointly, and have offspring distribution $F^\star$ for the second and further generations, and with offspring distribution given by $F$ for the first generation), as long as the total number of vertices of the explored clusters does not exceed $n^{\vr}$ for some small $\vr>0$.
Let us do the exploration in a breadth-first-search manner, and then the exploration at time $t$ contains all the vertices with at most distance $t$ away from $\CC^{\sss{(r)}}_0:=\{v_r\}$ and $\CC^{\sss{(b)}}_0:=\{v_b\}$. We shall denote these clusters by $\CC^{\sss{(r)}}_t, \CC^{\sss{(b)}}_t$ (i.e., the vertices and their graph structure), and think of them as being colored red and blue, respectively. Similarly, we denote the number of vertices in the $k$th generation of the coupled branching processes by $(Z_{k}^{\sss{(r)}}, Z_k^{\sss{(b)}})_{k>0}$.
The next definition, describing the \emph{double-exponential growth rates} of these neighborhoods, uses this coupling:
  \begin{definition}[Double-exponential growth rates of local neighborhoods]\label{def:limit-variables}
  Let $(Z_k^{\sss{(r)}}, Z_k^{\sss{(b)}})$ denote the number
 of individuals in the $k$th generation of the two independent copies of a Galton-Watson process, coupled to the breadth-first-search exploration process of the neighborhoods of $v_r$ and $v_b$ in the configuration model. In these branching processes, the size of the first generation has distribution $F$,
 and all further generations have offspring distribution $F^\star$ from Assumption \ref{assu:tv}.  Then, for some $\vr_n'<(\tau-2)\min\{ 
 {\beta_n}\kappa, (1-{\beta_n}(1+\ve))/2, (\tau-2-2\ve)/(2(\tau-1))\}$, let us define
\be\label{def:yrn-ybn}\Yrn:=(\tau-2)^{t(n^{\vr_n'})} \log (Z^{\sss{(r)}}_{t(n^{\vr_n'})}),  \quad \Ybn:=(\tau-2)^{ t(n^{\vr_n'})} \log (Z^{\sss{(b)}}_{ t(n^{\vr_n'})}),\ee
  where $t(n^{\vr_n'})=\inf_k\{\max\{ Z_k^{\sss{(r)}}, Z_k^{\sss{(b)}} \} \ge n^{\vr_n'}\}$.
Let us further introduce
  \be\label{def:Y} Y_r:= \lim_{k\to\infty} (\tau-2)^k \log (Z_k^{\sss{(r)}}), \quad Y_b:=\lim_{k\to\infty} (\tau-2)^k \log (Z_k^{\sss{(b)}}).\ee
\end{definition}
Note that the limit variables in \eqref{def:Y} are independent of $\rho_n'$.
 Further, note that $(\Yrn, \Ybn)$ is a subsequence of the convergent sequence $\left( (\tau-2)^k \log (Z_k^{\sss{(r)}}),  (\tau-2)^k \log (Z_k^{\sss{(b)}})\right)$, taken at the subsequence  $k_n:=t(n^{\rho_n'})$. Since
 for any $\rho'>0$, $t(n^{\rho_n'}) \to \infty$ as $n\to \infty$, under Assumption \ref{assu:tv} we shall obtain that $(\Yrn, \Ybn)\toindis (Y_r,Y_b)$ as $n\to \infty$.
When (and only when) ${\beta_n}=1/(\tau-1)$, we shall need one more assumption that concerns the limiting distribution of $Y_r, Y_b$:

\begin{assumption}[No pointmass of the measure of $Y$]\label{assu:pointmass}
We assume that the limiting random variable $Y:=\lim_{k\to\infty} (\tau-2)^k\log (\max\{ Z_k,1\})$ of the branching process in Definition \ref{def:limit-variables} has no point-mass on $(0,\infty)$.
\end{assumption}
The criteria on $F$ in Assumption \ref{assu:tv} required for this assumption to hold are not obvious. According to our knowledge, no necessary and sufficient condition for no point mass or absolute continuity can be found in the literature. A sufficient criterion for absolute continuity of $Y$ is given in \cite{Sene73, Sene74}.

 To be able to state the results shortly, let us define the $\sigma$-algebra generated by the induced subgraph on $\CC^{\sss{(r)}}_{t(n^{\vr_n'})} \cup \CC^{\sss{(b)}}_{t(n^{\vr_n'})}$:
 \be\label{eq:sigma-algebra} \CF_{\vr_n'}:=\sigma\left( \CC^{\sss{(r)}}_{t(n^{\vr_n'})} \cup \CC^{\sss{(b)}}_{t(n^{\vr_n'})}\right) \ee and introduce the shorthand notation\footnote{The notation comes from the fact that $\Yrn, \Ybn \in \CF_{\vr_n'}$.}
  \be\label{eq:cond-measure} \Pv_{\sss{Y}}(\cdot):= \Pv(\cdot| \CF_{\vr_n'}), \quad \Ev_{\sss{Y}}(\cdot):= \Ev(\cdot| \CF_{\vr_n'}) \ee
Further define, for $q\in \{r,b\}$,
\be\label{def:ti-bi} T_q({\beta_n}):=\left\lfloor\frac{\log\log (n^{\beta_n}) -\log(Y_q^{\sss{(n)}})}{|\log (\tau-2)|}\right\rfloor-1, \quad
b_n^{(q)}({\beta_n}):= \left\{\frac{\log\log (n^{\beta_n}) -\log(Y_q^{\sss{(n)}})}{|\log (\tau-2)|}\right\},\ee
where $\lfloor x\rfloor $ denotes the largest integer that is at most $x$ and $\{x\}= x-\lfloor x\rfloor$ denotes the fractional part of $x$. Further, let $\lceil x\rceil$ denote the smallest integer that is larger than $x$.
\subsubsection{Typical distances}

\begin{theorem}[Distances in truncated power-law configuration models]\label{thm:distances}
Consider the configuration model with empirical degree distribution satisfying Assumptions \ref{assu:degree-dist}-\ref{assu:tv} with some ${\beta_n}\in (0, 1/(\tau-1)]$ such that ${\beta_n}(\log n)^\gamma\to \infty$ for some $\gamma\in(0,1)$. When ${\beta_n}\to 1/(\tau-1)$ then we require that Assumption \ref{assu:pointmass} holds additionally. 
\be\label{eq:dist-beta-small-p2}
\mathrm d_G(v_r, v_b) =T_r({\beta_n})+T_b({\beta_n})+ \left\lceil    \frac{1/{\beta_n}-(\tau-2)^{b_n^{\sss{(r)}}({\beta_n})}-(\tau-2)^{b_n^{\sss{(b)}}({\beta_n})}}{3-\tau} \right\rceil +1.
\ee

\end{theorem}
The sequence $Y_q^{\sss{(n)}}$ in $T_q({\beta_n}), \bnq({\beta_n})$ converges in distribution as $n\to \infty$. It is straightforward to show that for a sequence of random variables $X_n$ converging in distribution, the transforms $\lfloor \log \log (n^{\beta})+X_n\rfloor$ and $\{\log \log n+X_n\}$ do not converge, since their distribution shifts along as $\log \log (n^{\beta})$ moves from one integer to the next. 
We can rephrase the statement of Theorem \ref{thm:distances} in terms of convergence in distribution by filtering out the parts that show loglog-periodicity.
\begin{corollary}\label{corr:convergence}
The following distributional convergence holds:

\be\label{eq:distance-converge-p2}\begin{split}
	\mathrm d_G(v_r, v_b)- \frac{2\log\log (n^{\beta_n})}{|\log (\tau-2)|} -\left\lceil    \frac{1/{\beta_n}-(\tau-2)^{b_n^{\sss{(r)}}({\beta_n})}-(\tau-2)^{b_n^{\sss{(b)}}({\beta_n})}}{3-\tau} \right\rceil\\+b_n^{(r)}({\beta_n})+b_n^{(b)}({\beta_n}) 	 \toindis -1+ \frac{-\log  (Y_rY_b)}{|\log (\tau-2)|}.
\end{split}\ee
Alternatively, we obtain weak convergence along (double-exponentially growing) subsequences $(n_k)_{k\in \N}$ satisfying $\log \log (n_k^{\beta_n})=k+c+o(1)$ for every $c\in[0,1)$.
\end{corollary}

\begin{remark}\normalfont The criterion that ${\beta_n}(\log n)^\gamma\to \infty$  for some $\gamma\in (0,1)$ is just slightly stronger than requiring that the empirical second moment of the degrees in the graph  tends to infinity. Indeed, we give in \eqref{eq:second-mom} in Claim \ref{cl:tech} the upper bound $n^{(3-\tau){\beta_n}}$ on the empirical second moment of the degrees. A similar lower bound can be proved as well. This expression tends to infinity whenever ${\beta_n}\log n\to \infty$.
\end{remark}

\begin{remark}\label{rem:dominant}\normalfont
The message of Theorem \ref{thm:distances}  is that typical distances are centered around approximately $2\log\log (n^{{\beta_n}})/|\log(\tau-2)|+1/({\beta_n}(3-\tau))$, when the truncation of the degrees happens at a value $n^{{\beta_n}}$, with tight fluctuations around this value.
 As pointed out before, the threshold for the dominance of the two terms is at ${\beta_n}=\Theta(1/\log\log n)$. Indeed, as soon as ${\beta_n}=o(1/\log\log n)$, the term containing $1/{\beta_n}$ in \eqref{eq:distance-converge-p2} becomes dominant, and the second order term\footnote{Note that in \eqref{eq:distance-converge-p2}, the terms containing $\log\log n$ and $-\log {\beta_n}$ both compete to be the second order term, but under the criterion on ${\beta_n}$, $-\log {\beta_n}=O(\log\log n)$.} is of order $\log \log n$. On the other hand, when ${\beta_n} \log \log n\to \infty$, the dominant term is $\log\log n$.
\end{remark}

\begin{remark}[Dropping the condition on the minimal degree]\label{rem:no-mindeg}\normalfont
With slightly more work it is possible to drop Assumption \ref{assu:mindeg} from the assumptions in Theorems \ref{thm:distances}. Under Assumption \ref{assu:degree-dist} but without Assumption \ref{assu:mindeg}, the graph has a unique  \emph{giant component} of linear size, $\zeta n(1-o_{\Pv}(1))$ for some $\zeta>0$, see Janson and Luczak \cite{JanLuc09}.
In this case, the statement of Theorem \ref{thm:distances}  remain valid \emph{conditioned} on the event that both $v_r, v_b$ are in the giant component of the graph.
This conditioning can be done similarly as described in \cite{HHZ07}. To keep our paper short, we omit to provide the proof here, since this is not the main focus of this paper.
\end{remark}
\subsubsection{Structure and number of shortest paths}
The next two theorems are by-products of the proof of Theorem \ref{thm:distances}. They reveal the structure of shortest paths and thus shed light on the topology of the graph in more detail. Let us denote a path connecting vertices $u,v$ by $\CP_{u,v}$, and let us denote any path that realizes the graph distance $\mathrm d_G(u,v)$ by $\CP^\star_{u,v}$. We write $w\in \CP_{u,v}$ if $w$ is a vertex $\neq u,v$ that is on the path $\CP_{u,v}$.
We write $\mathrm d_G(u,v|\Lambda)$ for the length of the shortest  path between two vertices $u,v$ restricted to contain vertices in a set $\Lambda$. Finally, let us write
\be\label{eq:lambda-le}\Lambda_{\le z}:=\{v\in [n]: d_v\le n^z \},\ee
and for a triplet $(z, x_1,x_2)$ of numbers let us define the `upper' and `lower' fractional-part of the following expression:
 \begin{align}\label{def:c}f^u(z, x_1, x_2)&:=\left\lceil    \frac{1/z-x_1-x_2}{3-\tau} \right\rceil-  \frac{1/z-x_1-x_2}{3-\tau},\\
 \label{eq:polinomial} f^\ell(z, x_1, x_2)&:= \frac{1/z-x_1-x_2}{3-\tau} - \left\lfloor \frac{1/z-x_1-x_2}{3-\tau} \right\rfloor.
 \end{align}
 Note that either $f^u=f^\ell=0$ or $f^u=1-f^\ell\in(0,1)$.
\begin{theorem}[Structure and number of shortest paths between hubs]\label{thm:structure-1}
Under the same conditions as Theorem \ref{thm:distances}, let $\wit{\beta_n}\le{\beta_n}$ be such that $\wit{\beta_n} (\log n)^\gamma\to \infty$. Let $x_1,x_2 > \tau-2$ and $v_1,v_2$ be two vertices with degrees $d_{v_j}\sim n^{x_j\wit{\beta_n} }$ for $j=1,2$. Then, the distance between $v_1,v_2$ \emph{restricted to} paths $\CP_{v_1,v_2}$ that contain only vertices with degree at most $n^{\wit{\beta_n}}$ is whp
 \be\label{eq:alpha-truncation} \mathrm{d}_G(v_1,v_2\mid \Lambda_{\le\wit{\beta_n}})=\left\lceil    \frac{1/(\wit{\beta_n})-x_1-x_2}{3-\tau} \right\rceil +1,\ee
while the \emph{number of shortest paths}\footnote{Here, overlaps between paths are allowed. We consider two paths different if they have at least one different edge.} between $v_1, v_2$ within $\Lambda_{\le \wit{\beta_n}}$ satisfies whp
 \be\label{eq:nof-shortest} \#\{\CP_{v_1, v_2}^\star \mid \Lambda_{\le \wit{\beta_n}}\} \sim n^{\wit{\beta_n} f^u(\wit{\beta_n}, x_1, x_2)},\ee
 where $f^u(\wit{\beta_n}, x_1, x_2)$ is defined in \eqref{def:c}.
Further,
 whp all vertices on \emph{any} shortest path in $\Lambda_{\le \wit{\beta_n}}$ connecting $v_1,v_2$  have degree at least
 $n^{\wit{\beta_n} f^\ell(\wit{\beta_n}, x_1, x_2)}$. I.e., for all $\ve>0$,
\be\label{eq:no-low-degree}
\Pv\left( \exists \CP^\star_{v_1,v_2}\in \Lambda_{\le \wit{\beta_n}}, w\in \CP^\star_{v_1,v_2}: d_w \le n^{\wit{\beta_n} f^\ell(\wit{\beta_n}, x_1, x_2)(1-\ve)}\right) \to 0.
\ee
\end{theorem}
Interpreting Theorem \ref{thm:structure-1},
first note that setting $\wit{\beta_n}<{\beta_n}$ in \eqref{eq:alpha-truncation} reveals the distance between the two vertices when the path must avoid vertices with degree at least $n^{\wit{\beta_n}}$.
While
\eqref{eq:alpha-truncation} for $\wit{\beta_n}\equiv {\beta_n}$ shows that the generating function approximation known from physics - the formula described in \eqref{eq:physics} - is valid when instead of typical distances, we consider distances between very high-degree vertices.
The second statement, \eqref{eq:nof-shortest} shows that the number of shortest paths between hubs concentrate on a logarithmic scale, since the statement 
\be\label{eq:convergence-nof-paths} \frac{\log \left( \#\{\CP_{v_1, v_2} \in \Lambda_{\le \wit{\beta_n}}\}\right)}{ f^u(\wit{\beta_n}, x_1, x_2)\wit{\beta_n}\log n } \toinp 1.\ee
is a direct consequence of \eqref{eq:nof-shortest}.
This implies that there are \emph{many} shortest paths,
$n^{\wit{\beta_n} f^u(\wit{\beta_n}, x_1, x_2)}$ many. Note however that as soon as any of $\wit{\beta_n}, x_1, x_2$ depends on $n$, the upper fractional part $f^u$ starts to oscillate in the interval $[0,1)$.
Similarly, the statement in \eqref{eq:no-low-degree} shows that \emph{all these shortest paths} use relatively high degree vertices - just a factor $f^\ell$ multiplies  the exponent of the maximally allowed degree $n^{\wit{\beta_n}}$.  Keep in mind that $f^u$ and $f^\ell$ are does not necessarily are bounded away from $0$.

Note again that as soon as any of ${\beta_n}, x_1, x_2$ depends on $n$, $f^\ell({\beta_n}, x_1, x_2)$ oscillates together with this fractional part. 

Comparing \eqref{eq:nof-shortest} and \eqref{eq:no-low-degree}, one notes that $f^u+f^\ell=1$ unless both of them are $0$. One can avoid integer values by slightly changing ${\beta_n}$ or $x_1, x_2$ by pushing some $n$-dependent terms into the $d_{v_j}\sim n^{x_j {\beta_n}}$ relation. Further, it is not hard to extend the proof of Theorem \ref{thm:structure-1} to show that there is at least one shortest path that uses a vertex with degree $\sim n^{\wit{\beta_n} f^\ell(\wit \beta_n, x_1, x_2 )(1+\ve)}$ for arbitrary small $\ve>0$. 
 Thus, we arrive to the following observation.

\begin{observation}\label{obs:factor} Let $v_1, v_2$ be two hubs with $d_{v_j}\sim  n^{x_j{\beta_n}}$ for $x_j>\tau-2, j\in\{1,2\}$. Then
the number of shortest paths between $v_1, v_2$ times the lowest degree that these paths use
gives approximately\footnote{up to error terms of order at most $\exp\{\pm(\log n^{\beta_n})^\theta\}$, for some $\theta<1$.} the maximal degree in the graph, i.e., $\sim n^{\beta_n}$.\end{observation}
We provide a sketch proof of this observation in Section \ref{s:extension}.
Our next theorem analyses the number and structure of shortest paths between two uniformly chosen vertices $v_r, v_b$:
\begin{theorem}[Structure  and number of shortest paths]\label{thm:structure-2}
Under the same conditions as in Theorem \ref{thm:distances}, there is a shortest path between $v_r, v_b$ that has the following structure, whp:

\begin{enumerate}
\item[(1)]\emph{(Degree-increasing phase)} For both $q\in\{r,b\}$, starting from $v_q$, a path segment of length $T_q({\beta_n})=\log \log (n^{{\beta_n}})/|\log (\tau-2)|+\Theta_{\Pv}(1)$ as in \eqref{def:ti-bi} ends with a vertex $v_q^\star$ with degree
 \be\label{eq:degree-at-segment} d_{v_q^\star}\sim n^{{\beta_n} (\tau-2)^{\bnq({\beta_n})}},\ee where $\bnq({\beta_n})$ is from \eqref{def:ti-bi}. The vertex $v_q^\star$ can be chosen to be the maximal degree vertex among all vertices that are reachable from $v_q$ on a path of length $T_q({\beta_n})$.
  For any $k<i_{\star\sss{(q)}}({\beta_n})$, the degree of the $(T_q({\beta_n})-k)$th vertex on the path between $v_q, v_q^\star$  is  $\sim n^{{\beta_n} (\tau-2)^{\bnq({\beta_n})+k}}$, where $i_{\star\sss{(q)}}({\beta_n})$ are tight random variables given below in \eqref{eq:value_i*-al} and \eqref{eq:istar-b}.

\item[(2)]\label{phase-2}\emph{(Connection among high-degree vertices phase)} A path of length
\be\label{eq:connecting-segment} \left\lceil    \frac{1/{\beta_n}-(\tau-2)^{b_n^{\sss{(r)}}({\beta_n})}-(\tau-2)^{b_n^{\sss{(b)}}({\beta_n})}}{3-\tau} \right\rceil +1\ee connects $v_r^\star, v_b^\star$ using only vertices with degree at least $n^{{\beta_n} f_n^\ell}$, where,  in agreement with Theorem \ref{thm:structure-1}, $f^\ell_n:=f^\ell({\beta_n}, (\tau-2)^{\bnr({\beta_n})}, (\tau-2)^{\bnb({\beta_n})})$.
Further, Phase (2) is valid for all shortest paths, whp. That is, whp, for any shortest path $\CP_{v_r, v_b}^\star$, the segment between the $T_r({\beta_n})$th and the $|\CP_{v_r, v_b}^\star|-T_b({\beta_n})$th vertex has length as in \eqref{eq:connecting-segment}, and it only contains vertices with degree at least $n^{{\beta_n} f^\ell_n}$.
\end{enumerate}
Finally, as a consequence of Theorem \ref{thm:structure-1}, the number of shortest paths between $v_r, v_b$ satisfies $\Pv_{\sss{Y}}-whp$
\[ \#\{ \CP_{v_r, v_b}^\star \}\sim n^{{\beta_n} f^u_n},\]
where $f^u_n:=f^u({\beta_n}, (\tau-2)^{\bnr({\beta_n})}, (\tau-2)^{\bnb({\beta_n})} )\in[0,1)$ fluctuates with $n$.
\end{theorem}
Theorem \ref{thm:structure-2} sheds light on the true structure of the shortest path between two uniformly chosen vertices: both ends of the path start with a segment where after a short initial randomness, degrees are essentially increasing in a deterministic fashion: each degree is asymptotically a power $1/(\tau-2)$ of the previous degree on the path. This phase ends with a vertex $v_q^\star$ that has degree in the interval $[n^{(\tau-2){\beta_n}}, n^{{\beta_n}}]$, for $q\in\{r,b\}$.
The precise (random) prefactors of ${\beta_n}$ for $q\in\{r,b\}$ in the exponent of $n$ of the degree $d_{v_q^\star}$ determines the precise length of the next phase, that establishes a connecting path between $v_r^\star, v_b^\star$. For this path segment, Theorem \ref{thm:structure-1} can be applied, which means that   the length of this path is at most a constant $\lceil 2/(3-\tau)\rceil$ away from $1/({\beta_n}(3-\tau))$, and all vertices on this path have relatively high degree.

We can turn Observation \ref{obs:factor} to hold for two uniformly chosen vertices as well: the non-integer condition holds whp under Assumption \ref{assu:pointmass} with $(x_1, x_2)=((\tau-2)^{\bnq({\beta_n})})_{q\in\{r,b\}}$, so we arrive to the following observation:
\begin{observation}\label{obs:factor-2} Let $v_r, v_b$ be two uniformly chosen vertices. Then whp under Assumption \ref{assu:pointmass},
the number of shortest paths between $v_r, v_b$ times the lowest degree that these paths use  in Phase (2) in Theorem \ref{thm:structure-2}
is approximately the maximal degree in the graph, i.e., it is $\sim n^{\beta_n}$.
\end{observation}

\subsubsection{Attack vulnerability}
Let us mention that when we remove some set of vertices $\Lambda$ from a configuration model and the edges attached to them (called an attack), the remaining graph is still a configuration model on $[n]\setminus \Lambda$, with a new empirical degree distribution that might become \emph{random}, depending on the type of the attack.
Observe that the shortest path between two vertices in the remaining graph is the same as the shortest path in the original graph restricted to stay among vertices in $[n]\setminus \Lambda$.
When the attack is so that it removes all vertices above a certain degree, we call it a \emph{targeted attack} or \emph{deliberate attack}.
This is the meaning of setting $\wit\beta_n< \beta_n$ in Theorem \ref{thm:structure-1} above.

An immediate corollary of the proof of Theorems \ref{thm:distances} and \ref{thm:structure-1} is that our results remain valid in the configuration model with targeted attack as well, that is, when instead of truncating the degrees, we \emph{remove} all vertices with degree at least $n^{\wit{\beta_n}}$ from a configuration model. Equivalently, we can consider the length of the shortest path restricted to stay among vertices with degree at most $n^{\wit{\beta_n}}$.

\begin{corollary}\label{cor:attack}
Let us consider the configuration model under the same assumptions as the ones in Theorem \ref{thm:distances}. For a sequence $\wit{\beta_n}$ with $\wit{\beta_n} (\log n)^\gamma \to \infty$ for some $\gamma<1$, let us remove all vertices with degree at least $n^{\wit{\beta_n}}$ and the edges attached to them from the graph on $n$ vertices.
Then, the typical distance between two vertices $v_r, v_b$ chosen uniformly at random from the remaining set of vertices satisfies whp
\be\label{eq:distance-converge-p3}\begin{split}
	\mathrm d_G(v_r, v_b| \Lambda_{\le \wit{\beta_n}})- \frac{2\log\log (n^{\wit{\beta_n}})}{|\log (\tau-2)|} -\left\lceil    \frac{1/(\wit{\beta_n})-(\tau-2)^{b_n^{\sss{(r)}}(\wit{\beta_n})}-(\tau-2)^{b_n^{\sss{(b)}}(\wit{\beta_n})}}{3-\tau} \right\rceil\\+b_n^{(r)}(\wit{\beta_n})+b_n^{(b)}(\wit{\beta_n}) 	 \toindis -1+ \frac{ -\log  (Y_rY_b)}{|\log (\tau-2)|}.
\end{split}\ee
Further, Theorem \ref{thm:structure-2} also remains valid in this setting, with ${\beta_n}$ replaced by $\wit{\beta_n}$ everywhere.
\end{corollary}
This corollary sheds light on the effect of a targeted attack - commonly known as the \emph{attack vulnerability} of the network. In fact, Corollary \ref{cor:attack} describes the way typical distances grow when we (gradually) remove the `core' of the graph, meaning all vertices with degree at least $n^{\wit{\beta_n}}$. For example, starting with a configuration model with i.i.d.\ degrees, (corresponding to ${\beta_n}\equiv 1/(\tau-1)$), one has to go as far as to choose $\wit{\beta_n}=o(1/\log \log n)$ to change the order of magnitude of the length of shortest paths.

An alternative proof of this corollary could be the following: It can be shown that the number of vertices in $\Lambda_{\le \wit {\beta_n}}=o(n)$, so, only $o(n)$ many vertices are removed. Then, one can show that -- even though these are the highest degree vertices -- the total number of half-edges attached to the removed vertices is still $o(n)$. When considering the degree of a remaining vertex in the remaining graph, the half-edges that were matched to removed vertices should also be removed. This results in a \emph{thinning} of the degrees. This thinning is \emph{not independent} for different half-edges and vertices. However, by a stochastic domination argument one can still show that the resulting new degree distribution still satisfies the conditions of Theorem \ref{thm:distances}, with now ${\beta_n}$ replaced by $\wit{\beta_n}$.  This makes clear why all expressions in \eqref{eq:distance-converge-p3} depend only on $\wit{\beta_n}$ and not on the original ${\beta_n}$.

\subsection{Examples}\label{s:examples}
 Note that Assumption \ref{assu:degree-dist} is satisfied in the following cases that we keep in mind to study:

 \begin{example}\label{ex:iid}\normalfont The first example arises when the degrees are \emph{independent and identically distributed} from a background power-law distribution $F$ that satisfies \eqref{eq:F} and \eqref{eq:L} (for all $x \in \R$). In this case, it is not hard to see that the order of magnitude of the maximal degree in the graph is $n^{(1+o_{\Pv}(1))/(\tau-1)}$ whp.  Further, using the concentration of binomial random variables (see \cite{HHZ07} for the computations) shows that whp the empirical degree distribution satisfies Assumption \ref{assu:degree-dist}, with ${\beta_n}=1/(\tau-1)$, with a possibly larger constant $C$ in the bound on $L_n(x)$ than the one in the background distribution $F$.
 \end{example}
 Pure power-law degrees were found for example in the internet backbone network \cite{Falo1999}, in metabolic reaction networks \cite{Jeon2000}, in telephone call graphs \cite{Nana2006}, and most famously, in the world-wide-web \cite{Bara1999, Brod2000, Kuma1999}.
 \begin{remark}[Fluctuations of typical distances in the i.i.d.\ degree case]\normalfont
In the special case of i.i.d.\ degrees as in Example \ref{ex:iid}, the value ${\beta_n}=1/(\tau-1)$. Under Assumption \ref{assu:pointmass}, the upper integer part in \eqref{eq:dist-beta-small-p2} simplifies to either $0$ or $1$, and typical distances in this case become
\be  \label{eq:distance-iid}  \mathrm d_G(v_r, v_b) = T_r({\beta_n}) + T_b({\beta_n}) + 2 - \ind\{ \tau-1<(\tau-2)^{\bnr({\beta_n})} + (\tau-2)^{\bnb({\beta_n})}\} .\ee
 We emphasize that Theorem \ref{thm:distances} implies that the typical distances in the graph are concentrated around $2 \log\log n / |\log (\tau-2)|$ with bounded fluctuations, a result that already appeared in \cite{HHZ07} for the i.i.d.\ degree case. The statement of Corollary \ref{corr:convergence} `filters out' the bounded oscillations arising from fractional part issues
that oscillate with $n$. We emphasize here that the statement of Theorem \ref{thm:distances} applied to i.i.d.\ degrees and \cite[Theorem 1.2]{HHZ07} are essentially the same. However, they provide a different description of typical distances. The current proof here is much shorter than the one in \cite{HHZ07} as well as it allows to treat the truncated degree case at the same time. We show in Section \ref{s:compare} that the two theorems are indeed the same.
\end{remark}

\begin{example}[Exponential truncation]\label{ex:trunc}\normalfont The degrees are generated i.i.d. from an $n$-dependent \emph{truncated power law} distribution $F^{(n)}$ that can be written in the form
\be\label{eq:truncated-power}  1-F^{(n)}(x)=\frac{L^{(n)}(x)}{x^{\tau-1}} \exp\{ -c x / n^{\beta_n}\}, \ee
with $L^{(n)}$ satisfying \eqref{eq:L}.  In this case, the empirical distribution $F_n(x)$ satisfies Assumption \ref{assu:degree-dist} for all sufficiently large $n$, since for any $x\le n^{{\beta_n}(1-\ve)}$, the exponential term is at least $1/2$, say, while for any $x\ge n^{{\beta_n}(1+\ve)}$, $F^{(n)}(x)=O(1/n^2)$, thus we shall not actually see vertices with such high degree in the graph for large enough $n$.
A special case arises when $(d_i)_{i\in [n]}$ are i.i.d. with $d_i=\min\{ X_i, G_i\}$, where  $G_i$ are i.i.d.\ geometric random variables with mean $\exp\{ c/n^{\beta_n}\}$, and $X_i$ i.i.d.\ as in Example \ref{ex:iid}.
 \end{example}
 Power-law degrees with exponential truncation were proposed e.g. in \cite{Mossa2002}, and are observed for instance in the movie actor network \cite{AmaScaBarSta00}, air transportation networks \cite{Guim2005} and co-authorship networks \cite{Newman2001, Stro2001}, brain functional networks \cite{Acha2006}, ecological networks \cite{Mont2006} such as coevolutionary networks of plant-animal interactions \cite{Jord2003}.
 \begin{example}[Hard truncation]\label{ex:trunc-hard}\normalfont The degrees are generated i.i.d.\ from an $n-dependent$ \emph{truncated power-law} that can be written in the form
\be\label{eq:truncated-power}  F^{(n)}(x)=\frac{L^{(n)}(x)}{x^{\tau-1}} \ind_{x<n^{\beta_n}}, \ee
with $L^{(n)}$ satisfying \eqref{eq:L}.
A special case again arises when $(d_i)_{i\in [n]}$ are i.i.d.\ with $d_i=\min\{X_i, n^{\beta_n}\}$, where $X_i$ are i.i.d. as in Example \ref{ex:iid}.
\end{example}
Probably the most important example for hard truncation is a \emph{targeted attack}, since in this case every vertex above a certain degree in the network is removed. Scale-free graphs are often called attack vulnerable, see e.g. \cite{AlBarHaw00, CohEre01, HolKim02}.
A theoretical example where the authors use a network model with hard truncation can be found in \cite{GaoWor2016} or \cite{McKay1991}.

Perhaps surprisingly, the online social network of Facebook does not seem to follow a truncated power-law \cite{Ugan2011}, even though the total number of friends of a person was limited to $5000$ at the time of the measurement.
Nevertheless, we would like to emphasize that our theorem allows for many possible truncations functions, among which the hard truncation is possibly the most strict.

\subsection{Discussion and open questions}\label{s:discussion}
\subsubsection*{Heuristic explanation of the formula in Theorem \ref{thm:distances}}
In Theorem \ref{thm:distances}, we have determined that the distances are centered around
\[ \frac{2 \log\log (n^{{\beta_n}})}{|\log (\tau-2)|} + \frac{1}{{\beta_n}(3-\tau)}  \]
with tight fluctuations.
Here we give a heuristic explanation of this formula.
For graphs with locally tree-like structure, the usual BP approximation says that the number of vertices that are reachable on a path of length $k$ from $v_r$,
is approximately $Z_k^{\sss{(n)}}$, the size of generation $k$ of a BP with offspring $D_n$ in the first and $D_n^\star$ in the consecutive generations.
Generating function methods then yield that $\Ev[Z_k^{\sss{(n)}}]=\Ev[D_n] \nu_n^{k-1}$, with $\nu_n=\Ev[D_n^\star]\sim n^{{\beta_n}(3-\tau)}$ as in \eqref{eq:edn} below, and then the approximation $Z_k^{\sss{(n)}} \approx \Ev[Z_k^{\sss{(n)}}]$ is often used. Unfortunately, for heavily skewed distributions like that of $D_n^\star$, it does not hold that $Z_k^{\sss{(n)}}\approx \Ev[Z_k^{\sss{(n)}}]$. This is so because $\Ev[D_n^\star]$ is characterised by the highest degree vertices, of degree $n^{{\beta_n}}$, while, on the other hand, low degree vertices are typically not connected to these hubs in the graph, and thus $Z_k^{\sss{(n)}} \ll \Ev[Z_k^{\sss{(n)}}]$.

It is true however that
\be\label{eq:approx-1}Z_k^{\sss{(n)}}\approx C^{1/(\tau-2)^k},\ee for some random constant $C$ \cite{D78}. Thus, as long as the BP approximation is valid, we see a `degree-increasing phase' within the exploration clusters of the vertices $v_r,v_b$. From the approximation \eqref{eq:approx-1} it already follows that it takes $\log \log (n^{{\beta_n}})/|(\log(\tau-2))|+$tight number of steps to reach a hub $v_q^\star$ in the graph, for $q\in\{r,b\}$. Extreme value theory tells us that any hub will have some neighbors that are also hubs, and thus the approximation that $D_n^\star\approx \Ev[D_n^\star]$ and consequently $Z_k^{\sss{(n)}}\approx \Ev[Z_k^{\sss{(n)}}]$ suddenly becomes valid when considering the number of vertices of distance $k$ away from $v_q^\star$.
This means that it takes an \emph{additional} $\log n/\log\nu_n=1/{\beta_n}(3-\tau)+$tight number of steps to connect the two hubs $v_r^\star, v_b^\star$ to each other. This explains the formula in Theorem \ref{thm:distances}.

 \subsubsection*{Comment about `structural cutoff'}  Often in physics literature, ${\beta_n}=1/2$ is called `structural cut-off' \cite{Osti2014}.
 When ${\beta_n}>1/2$, vertices with degree at least $n^{1/2+\ve}$ form a complete subgraph of the graph, while for ${\beta_n}<1/2$ this complete subgraph is not present. Further, when ${\beta_n}>1/2$,   a growing number of multiple edges appears, while for ${\beta_n}<1/2$, the number of multiple edges in the graph stays bounded. 
Our theorems show that there is no significantly different behavior of typical distances when the truncation happens below versus  above the structural cutoff $n^{1/2}$.

\subsubsection*{Open questions}
We believe the criterion ${\beta_n} (\log n)^\gamma\to \infty$ for some $\gamma\in (0,1)$ can be relaxed to be ${\beta_n} \log n\to \infty$, at least when one imposes more strict bounds on the slowly varying function $L_n(x)$ in \eqref{eq:F}. The criterion ${\beta_n} \log n\to \infty$ is the weakest form that is necessary for the empirical second moment to tend to infinity. Provided one can generalize our results to hold whenever ${\beta_n} \log n\to \infty$, we obtain a perfect interpolation between doubly logarithmic and logarithmic distances.
When ${\beta_n} \log n=\theta(1)$, the empirical second moment remains bounded and thus a finite mean BP approximation becomes available.

\subsubsection*{Notation}
We write $[n]$ for the set of integers $\{1,2,\dots, n\}$.
As usual, we write i.i.d.\ for independent and identically distributed, lhs and rhs for left-hand side and right-hand side.  We use $\toindis, \toinp, \toas$ for convergence in distribution, in probability and almost surely, respectively. We use the Landau symbols $o(\cdot), O(\cdot), \Theta(\cdot)$ in the usual way. For sequences of random or deterministic variables $X_n, Y_n$ we further write $X_n = o_{\Pv}(Y_n)$ and $X_n = O_{\Pv}(Y_n)$ if the sequence $X_n/Y_n \toinp 0 $ and is tight, respectively.

Constants are typically denoted by $c$ in lower and $C$ in upper bounds (possible with indices to indicate which constant is coming from which bound), and their precise values might change from line to line.
 We introduce  $(\CC^{\sss{(r)}}, \CC^{\sss{(b)}}):=(\CC^{\sss{(r)}},\CC^{\sss{(b)}})$.
 At some time $t$ along the matching or the exploration, for a set of vertices $\CA_t$ we denote the set of unpaired half-edges at that moment attached to vertices in $\CA_t$ as $\CH(\CA_t)$ and its size by $H(\CA_t)$. Thus $H(\CA_t)=\sum_{v\in \CA_t} \sum_{s \text{ half-edge attached to v}} \ind_{s \text{ not paired yet}}$. When the time is set to be $0$, $H(\CA)=\sum_{v\in\CA} d_v$. We write $d_v$ for the degree of vertex $v$.

\subsection{Overview of the proof}
To determine the distance between $v_r, v_b$, we start growing two clusters that we call red and blue, respectively, in a breadth-first-search manner, and see how these two clusters reach the highest degree vertices. As long as the two clusters are disjoint, the growth is not necessarily simultaneous, i.e., we might stop the growth of one color earlier.
To describe the growing clusters, we extensively use the fact that in the configuration model half-edges can be paired in an arbitrarily chosen order. This allows for a joint construction of the graph together with the growing of the two colored clusters. 
In Proposition \ref{prop:hubs} (Section \ref{s:hubs}) below, we show that the highest degree vertex $v_q^\star$, $q\in\{r,b\}$ that is reached by this path is of degree
\be\label{eq:dvq-1} d_{v_q^\star}\sim n^{{\beta_n}(\tau-2)^{b_n^{(q)}({\beta_n})}},\ee for $q\in\{r,b\}$, respectively.  The total length of this path from $v_q$ is $T_q({\beta_n})$ as in \eqref{def:ti-bi} for $q\in\{r,b\}$. The proof of this proposition has the following ingredients:
We couple the initial stages of the growth to two independent branching processes (BPs) (Section \ref{s:couple}). The coupling fails when  one of the colors (wlog we assume it is red)
    reaches  size $n^{\vr_n'}$ for some $\vr_n'>0$ sufficiently small.  From the half-edges attached to the BP cluster of $v_q, q\in\{r,b\}$, we build a path through higher and higher degree vertices to a vertex with degree at least $\sim n^{(\tau-2){\beta_n}}$.
In Lemma \ref{lem:badpaths}  we give an upper bound on the degree of the maximal-degree vertex reached at any time $t(n^{\vr_n'})+i$ of the exploration, implying that in $T_q({\beta_n})$ hops no vertex of degree higher than $\sim n^{{\beta_n}(\tau-2)^{b_n^{(q)}({\beta_n})}}$  is reached from $v_q$. Thus, this lemma serves also as a building block for the proof of Proposition \ref{prop:hubs}, but beyond that, it also enables us to show two other important things, that form the content of  Proposition \ref{lem:no-early-meeting} (Section \ref{s:no-early}):\\
(1) An early meeting is highly unlikely, i.e., the clusters $\CC^{\sss{(q)}}_{T_q({\beta_n})}$ are $\Pv_{\sss{Y}}$-whp disjoint. \\
(2) The quantity in \eqref{eq:dvq-1}  bounds also the \emph{total number of half-edges} attached to the explored cluster $\CC^{\sss{(q)}}_{T_q({\beta_n})}$.

 Finally, in Section \ref{s:main-proof} we finish the proof of Theorem \ref{thm:distances}. For the lower bound, we count the number of $z$-length paths between the two disjoint clusters $\CC^{\sss{(q)}}_{T_q({\beta_n})}$. Here, we use a first moment method (i.e., we show that the expected number of paths is $o(1)$) when $z$ is less than  the expression in \eqref{eq:connecting-segment}. For the upper bound, we establish the existence of a path of length as in \eqref{eq:connecting-segment} between $v_r^\star, v_b^\star$. We do this using a second moment method. This completes the proof of Theorem \ref{thm:distances}. We prove Theorems \ref{thm:structure-1} and \ref{thm:structure-2} in Section \ref{s:extension}. In Section \ref{s:compare} we compare Theorem \ref{thm:distances} in the special case $\beta_n\equiv 1/(\tau-1)$ to the result in \cite{HHZ07}.

\section{Distance from the hubs}\label{s:hubs}
In this section we analyse the distance of $v_r, v_b$ from the highest-degree vertices. The construction is similar to that of \cite[Section 3]{BarHofKom14}, however, since Assumption \ref{assu:degree-dist} on $F_n$ is weaker than the one in \cite{BarHofKom14}, we need to modify the proof. Recall that we
call the vertices with degree at least $n^{{\beta_n}(\tau-2)}$ hubs and recall $T_q({\beta_n}), \bnq({\beta_n})$ from \eqref{def:ti-bi}. More precisely, let us define
\be\label{eq:hubs} \mathrm{hubs}:=\{ v\in[n]: d_v\ge n^{{\beta_n}(\tau-2)}\},\ee
and for a set of vertices $\mathcal A \subseteq[n]$ and a vertex $v\in[n]$,
\be\label{eq:dist-set} \mathrm d_G(v, \mathcal A):= \min_{a\in \mathcal A} \mathrm d_G(v, a).\ee
The next proposition determines the distance of the uniformly chosen vertices $v_r, v_b$ from the hubs:
\begin{proposition}[Distance from the hubs]\label{prop:hubs}
Let us consider the configuration model on $n$ vertices with empirical degree distribution that satisfies Assumption \ref{assu:degree-dist}, and let $v_r, v_b$ be two uniformly chosen vertices. Then, for $q\in\{r,b\}$, $\Pv_{\sss{Y}}$-whp,
\[ \mathrm d_G(v_q, \mathrm{hubs})=T_q({\beta_n})=\frac{\log\log n+\log \left({\beta_n}/Y_n^{\sss{(q)}}\right)}{|\log(\tau-2)|} -1-\bnq({\beta_n}).\]
More precisely, $\Pv_{\sss{Y}}$-whp there is a vertex $v^\star_q\in \mathrm{hubs}$ at distance $T_q({\beta_n})$ away from $v_q$ with degree
\be \label{eq:degree-hub} d_{v_q^\star}\sim n^{{\beta_n}(\tau-2)^{\bnq({\beta_n})}}, \ee
while all vertices at distance at most $T_q({\beta_n})-1$ from $v_q$ are not hubs.
\end{proposition}
The main goal of this section is to prove this proposition. To show the upper bound, the proof has two main steps: the initial stage of the breadth-first-search  exploration (BFS) is coupled to branching process trees (Section \ref{s:couple}), while the later stage uses a decomposition of the vertices with degrees that are polynomial in $n$ into shells (Section \ref{s:shells}). To show the lower bound, we provide an upper bound on the degrees reached by the BFS in any shell at the time of first reaching that particular shell, see Lemma \ref{lem:badpaths} (Section \ref{s:badpaths}). This method is novel compared to the one in \cite{HHZ07}.

\subsection{Coupling the initial stages of BFS to branching processes}\label{s:couple}
In this section we investigate the initial stage of the spreading cluster of $v_r, v_b$. In the construction of the configuration model, at any time that we construct the matching of the half-edges, we are allowed to choose one of the not-yet-paired half-edges arbitrarily, and pair it to a uniformly chosen other not-yet-paired half-edge. Hence, we can do the pairing in an order that corresponds to the breadth-first-search (BFS) exploration started from $v_r, v_b$. That is, first we pair all the outgoing half-edges from the sources $v_r, v_b$ (distance $1$), then we pair the outgoing half-edges from the neighbors of the source vertices (distance $2$), and so on, in a breadth-first-search manner. Whenever we finish pairing all the half-edges attached to vertices at a given graph distance from the source vertices, we increase the distance by $1$. This process of joint construction of the BFS exploration and graph building is often called the \emph{exploration process} in the literature.

Recall that $\CC^{\sss{(r)}}_t, \CC^{\sss{(b)}}_t$ denotes the subgraph that is at distance at most $t$ from $v_r, v_b$.
 \cite[Proposition 4.7]{BHH10} (see also \cite[Lemma 2.2]{BarHofKom14}) shows that that the number of vertices and their forward degrees\footnote{Forward degree means the number of newly available half-edges when a new vertex is discovered, that will be paired to new vertices later on.} in the exploration process can be coupled to i.i.d.\ degrees having distribution function $F_n^\star$ from \eqref{def:size-biased1}, as long as the total number of vertices of the colored clusters is not too large. There, a different assumption is posed on the maximal degree in the graph, so we shortly adjust the proof of \cite[Proposition 4.7]{BHH10} to our setting below in Lemma \ref{lem:couple}.
 The distribution $F^\star_n$ arises from the fact that as long as the set of explored vertices is relatively small, a forward degree $j$ is generated when the uniformly chosen half-edge of a pairing belongs to a vertex with degree $j+1$.  The probability of choosing a half-edge that belongs to a vertex with degree $j+1$ is approximately equal to $(j+1)\Pv(D_n=j+1)/\Ev[D_n]$, and thus $F^\star_n$ and thus $F^\star$ are the natural candidates for the forward degrees in the exploration process.

 \begin{lemma}[Coupling error of the forward degrees]\label{lem:couple} Consider the configuration model with degree sequence that satisfies Assumption \ref{assu:degree-dist}. Then, in the exploration process  started from two uniformly chosen vertices $v_r, v_b$, the forward degrees
 $(X_k^{\sss{(n)}})_{k\le s_n}$ of the first $s_n$ newly discovered vertices can be coupled to an i.i.d. sequence $D_{n,k}^\star$ from distribution $D_n^\star$ with the following error bound
 \be\label{eq:couple-error-0} \Pv( \exists k\le s_n,  D_{n,k}^\star \neq X_k^{\sss{(n)}} ) \le C s_n^2 n^{{\beta_n}(1+\ve)-1}+C n^{- (\tau-2-2\ve)/(\tau+\ve)} s_n^{2(\tau-1+\ve/2)/(\tau+\ve)} \ee
 If further Assumption \ref{assu:tv} holds, then there is a coupling of $(X_k^{\sss{(n)}}, D_{n,k}^\star, D_k^\star)_{k\le s_n}$ with
 
 \be\ba\label{eq:couple-error-0tv} \Pv( \exists k\le s_n,  D_{n,k}^\star \neq X_k^{\sss{(n)}} \text{\,or\,} D_{n,k}^\star \neq D_k^\star) &\le  s_n n^{-{\beta_n}\kappa}+  C s_n^2 n^{{\beta_n}(1+\ve)-1}\\
 &\ \  +C n^{- (\tau-2-2\ve)/(\tau+\ve)} s_n^{2(\tau-1+\ve/2)/(\tau+\ve)}.\ea\ee
  \end{lemma}
 By choosing $s_n$ in Lemma \ref{lem:couple} so that the rhs of the bound in \eqref{eq:couple-error-0tv} still tends to zero, we obtain the following corollary:
  \begin{corollary}[Whp coupling of the exploration to two BPs]\label{corr:couple}
  In the configuration model satisfying Assumptions \ref{assu:degree-dist} and  \ref{assu:tv}, let $t$ be such that 
  \be\label{eq:snmax} |\CC^{\sss{(r)}}_t\cup \CC^{\sss{(b)}}_t|\le \min\{n^{{\beta_n}(\kappa-\delta)}, n^{(1-{\beta_n}(1+\ve)-\delta)/2}, n^{-(\tau-2-\ve-\delta)/(2(\tau-1+\ve/2))}\} \ee for some $\delta>0$. Then $(\CC^{\sss{(r)}}_t, \CC^{\sss{(b)}}_t)$ can be whp coupled to two i.i.d.\ BPs  with generation sizes $(Z_{k}^{\sss{(r)}}, Z_{k}^{\sss{(b)}})_{k>0}$  with distribution $F^\star$ for the offspring in the second and further generations, and with distribution $F$ for the offspring in the first generation.

  \end{corollary}

\begin{proof}[Proof of Lemma \ref{lem:couple}]

We would like to couple the forward degrees $(X_k^{\sss{(n)}})_{k\le s_n}$ in the exploration to an i.i.d. sample of size $s_n$ from distribution $D_{n}^\star$ as in \eqref{eq:size-biased2} as well as to $D^\star$ and estimate the coupling error.
The idea of the coupling is to achieve size-biased sampling with and without replacement of the vertices at the same time: this is \cite[Construction 4.2]{BHH14} that we informally recall here. Let us write $\CL_n$ for the list of half-edges.

We use a sequence of uniform random variables $U_k\in[0,1]$.
Then, we sample a uniform half-edge $h_k$ from $\CL_n$, namely we set $h_k$ to be the $j$th element of $\CL_n$ if $U_k\in((j-1)/\ell_n, j/\ell_n]$.
We sample the i.i.d.\ $D_{n,k}^\star$ by setting $D_{n,k}^\star$ to be $d_{v(h_k)}-1$, where $v(h)$ denotes the vertex that $h$ is incident to.

At the same time, we keep a list of already sampled vertices $\CS_k:=\{v_r, v_b, v(h_1),\dots, v(h_k)\}$. As long as $v(h_{k}) \notin \CS_{k-1}$, we can set $B_k^{\sss{(n)}}:=d_{v(h_k)}-1$, this quantity describes the number of brother half-edges of a newly discovered vertex via a pairing. Note that in this sampling  procedure there is no pairing of the half-edges yet. To ensure that the exploration cluster is a tree, at each step $k$ we yet have to check if any of these $B_k^{\sss{(n)}}=d_{v(h_k)}-1$ half-edges create cycles when being paired, i.e., they shall be paired to vertices in $\CS_{k-1}$. We write $X_k^{\sss{(n)}}$ to be $B_k^{\sss{(n)}}$ minus those half-edges that shall be paired to vertices in $\CS_{k-1}$.

Thus, the coupling to a BP tree with offspring distribution $D_n^\star$ can fail in two ways: either $v(h_{k}) \in \CS_{k-1}$ and the coupling between $B_k^{\sss{(n)}}, D_{n,k}^\star$ fails (depletion-of-points effect), or $X_k^{\sss{(n)}}<B_k^{\sss{(n)}}$ and some of the $B_k^{\sss{(n)}}$ half-edges create cycles (cycle-creation effect).

Introducing the $\sigma$-algebra $\CG_k$ generated by $v_r, v_b$ and the first $k$ draws, \cite[Lemma 4.3]{BHH14} bounds the coupling error between $B_k^{\sss{(n)}}, D_{n,k}^\star$: \be\label{eq:deplete} \Pv(B_k^{\sss{(n)}} \neq D_{n,k}^\star \mid \CG_{k-1}) \le \frac{1}{\ell_n} \left(d_{v_r} + d_{v_b} + \sum_{s=1}^{k-1} (B_s^{\sss{(n)}}+1)\right), \ee
while \cite[Lemma 4.3]{BHH14}  estimates the probability of creating a cycle at step $k$:
\be\label{eq:cycle}\Pv(X_k^{\sss{(n)}} < B_{k}^{\sss{(n)}} \mid \CG_{k-1}) \le \frac{B_k^{\sss{(n)}}}{\ell_n-1-(d_{v_r} + d_{v_b} + \sum_{s=1}^k (B_s^{\sss{(n)}}+1))} \left(d_{v_r} + d_{v_b} + \sum_{s=1}^k (B_s^{\sss{(n)}}+1)\right). \ee
Under Assumption \ref{assu:degree-dist} the maximal degree is at most  $n^{{\beta_n}(1+\ve)}$ in the graph, and $\ell_n=\Ev[D_n]n$ is of order $n$.
Combining these observations,  \eqref{eq:deplete} is at most $C k n^{{\beta_n}(1+\ve)-1}$.
Summing this bound over $k\le s_n$, we obtain that
\be\label{eq:couple-error-1} \Pv( \exists k\le s_n, B_k^{\sss{(n)}} \neq D_{n,k}^\star ) \le C s_n^2 n^{{\beta_n}(1+\ve)-1}. \ee
Then, on the event $\{\forall k\le s_n: B_k^{\sss{(n)}} = D_{n,k}^\star\}$, $B_s^{(n)}$ in \eqref{eq:cycle} can be replaced by the i.i.d.\ $D_{n,s}^{\star}$. Taking expectations of the rhs  \eqref{eq:cycle} does not work, since $\Ev[D_n^\star]$ is infinite. Thus we apply a truncation argument. Take $\ve>0$ so that $s_n^{(\tau-1+\ve)/(\tau-2)}=o(n)$. Then the denominator on the rhs of \eqref{eq:cycle} is at least $cn$ for some $c>0$. Further, for some truncation value $K_n$ to be chosen later that satisfies that $s_nK_n=o(n)$, 
\be\label{eq:couple-error-1}\ba \Pv( \exists k\le s_n, X_k^{\sss{(n)}} \neq B_{k}^{\sss{(n)}} \mid \forall k \le s_n, B_{k}^{\sss{(n)}} &=D_{n,k}^\star ) \le  \Pv(\exists j\le s_n, D_{n,j}^\star > K_n) \\&\ + \frac{1}{c n}\sum_{k=1}^{s_n}\sum_{s=1}^{k}\Ev\left[D_{n,k}^\star D_{n,s}^\star \ind_{\{\forall j\le s_n: D_{n,j}^\star \le K_n\}}\right]. \ea\ee
where we used that on the event that $\{ \forall k \le s_n: B_{k}^{\sss{(n)}} =D_{n,k}^\star,  D_{n,k}^\star \le K_n, \}$ the denominator on the rhs of \eqref{eq:cycle}  is at least $cn$ for some $c>0$. 
 Using \eqref{eq:size-biased2}, the first term on the rhs is at most $s_n L_n^\star(K_n )K_n^{-(\tau-2)}\le s_n K_n^{-(\tau-2)+\ve}$ for arbitrarily small $\ve>0$ for sufficiently large $n$, while the second term is at most $C n^{-1} K_n^2 s_n^2$. Making the order of the two error terms to be equal yields that the best choice of truncation value is at
 \be\label{eq:Kn} K_n:=(n/s_n)^{1/(\tau+\ve)}. \ee Note that the initial criterium that $s_nK_n=o(n)$ is satisfied for all $s_n=o(n)$, while $K_n$ is a polynomial of $n$ 
with strictly positive exponent whenever   $ s_n=o(n^{(1-\ve)/(\tau-1)})$. With $K_n$ in 
\eqref{eq:Kn}, the sum of the error terms in \eqref{eq:couple-error-1} becomes
\be \Pv( \exists k\le s_n, X_k^{\sss{(n)}} \neq B_{k}^{\sss{(n)}} \mid \forall k \le s_n, B_{k}^{\sss{(n)}} =D_{n,k}^\star )  \le C n^{- (\tau-2-2\ve)/(\tau+\ve)} s_n^{2(\tau-1+\ve/2)/(\tau+\ve)}\ee
which tends to zero as long as $s_n=o(n^{(\tau-2-\ve)/2(\tau-1+\ve/2)})$ for some arbitrarily small $\ve>0$.

Note that we have neglected coupling the forward degree of $v_r, v_b$ to two i.i.d. copies of $D_n$. This coupling can be done in a very similar way, by choosing with and without replacement two uniform numbers from $[n]$, with a coupling error at most $1/n$, which is negligible compared to the rhs of \eqref{eq:couple-error-1}. The probability that any of the $d_{v_r}$ or $d_{v_b}$ half-edges form cycles is at most of order $n^{{\beta_n}(1+\ve)}/\ell_n$, which again can be merged into the rhs of \eqref{eq:couple-error-1}.
This finishes the proof of \eqref{eq:couple-error-0}.

Next we extend the coupling between $(X_k^{\sss{(n)}}, D_{n,k}^\star)$ to additionally couple  $D^\star_k$ to them, using Assumption \ref{assu:tv}.
On the event that $X_k^{\sss{(n)}}= D_{n,k}^\star=\ell$, we use the optimal coupling that realizes the total variation distance between $D_n^\star$ and $D^\star$. Namely,
\be\label{eq:couple-def} \Pv(D^\star_k=\ell \mid X_k^{\sss{(n)}}= D_{n,k}^\star=\ell ):= \frac{ \min\{ \Pv(D^\star=\ell), \Pv(D_n^\star=\ell)\}}{\Pv(D_{n}^\star=\ell)}=\min\left\{1, \frac{\Pv(D^\star=\ell)}{\Pv(D_n^\star=\ell)}  \right\}.
\ee
One can set the other possible values of $D_k^\star$ as described e.g. in \cite[Chapter 1]{LevPerWil09}.
Nevertheless, the coupling error equals
\[ \Pv(D^\star_k\neq D_{n,k}^\star)=\sum_{\ell\ge1}(1-\min\{ \Pv(D^\star=\ell), \Pv(D_n^\star=\ell)\}) = d_{\sss{\mathrm{TV}}}(F^\star, F_n^\star).  \]
One can realize this coupling by using another independent uniform variable $\wit U_k$ to set the value of $D_k^\star$ once the value of $D_{n,k}^\star$ is determined. Thus, we obtain that
\be\label{eq:couple-error-2} \Pv(\exists k\le s_n, D_{n,k}^\star\neq D_{k}^\star)\le d_{\sss{\mathrm{TV}}}(F_n^\star, F^\star) \le s_n n^{-\kappa}.\ee
Similarly, we can couple $d_{v_r}, d_{v_b}$ to two i.i.d. copies of $D$ with a coupling error at most $2d_{\sss{\mathrm{TV}}}(F_n, F)$.
This finishes the proof of \eqref{eq:couple-error-0tv}.
\end{proof}

A theorem by Davies \cite{D78} describes the growth rate of a branching process with a given offspring distribution $G$ that satisfies \eqref{eq:size-biased2} that we describe here informally.  Let $\wit Z_k$ denote the $k$-th generation of a branching process with offspring distribution given by a distribution function $G$, that can be written in the form\footnote{Actually the theorem by Davies in \cite{D78} is somewhat more general, since it allows for a slightly larger class of slowly-varying functions than the criterion in \eqref{eq:L}. Nevertheless, the theorem applies for the case described in \eqref{eq:size-biased2} and \eqref{eq:L}.} as in \eqref{eq:size-biased2}  and \eqref{eq:L} for some $\tau\in(2,3)$ and some $x_0>0$ for all $x\ge x_0$.
Then $(\tau-2)^{k}\log(\wit{Z}_{k}\vee 1)$ converges almost surely to a random variable $\wit Y$. Further, the variable $\wit Y$ has exponential tails: if $J(x):=\Pv(\wit Y\le x)$, then
\be\label{eq:exp-tails-Y}   \lim_{x\to \infty} \frac{- \log (1-J(x))}{x} =1.\ee

We can apply Davies' theorem to each subtree of the two roots of the two i.i.d. BPs from Corollary \ref{corr:couple}, to obtain that the corresponding convergence in \eqref{def:Y} in Definition \ref{def:limit-variables}. See \cite[Lemma 2.4]{BarHofKom14} for more details.

Recall from Corollary \ref{corr:couple} that the coupling of the forward degrees in the BP and in the exploration fails when the total size of the BPs is too large, i.e., when \eqref{eq:snmax} is not satisfied.
Thus let us set\footnote{The only reason to denote this quantity by $\vr_n'$ is to be consistent with the notation in \cite{BarHofKom14}.} 
\be\label{eq:vrn} \vr_n':=(\tau-2)\min\{ {\beta_n}\kappa, (1-{\beta_n}(1+\ve))/2, (\tau-2-\ve)/(2(\tau-1+\ve/2))\}.\ee 
Without loss of generality we assume that the cluster of $v_r$ reaches size $n^{\vr_n'}$ first (otherwise we switch the indices $r,b$). Thus, let us define
\be\label{eq:stop} t(n^{\vr_n'})=\inf \{k:  Z_k^{\sss{(r)}}  \ge n^{\vr_n'} \}.\ee
From the definition \eqref{def:yrn-ybn}, an elementary rearrangement yields that (conditioned on $\Yrn$),
\be\label{eq:an2} t(n^{\vr_n'}) = \frac{\log(\vr_n'/\Yrn) + \log\log n}{|\log(\tau-2)|}+ 1-a_n^{\sss{(r)}},\ee
where
 \be\label{eq:an} a_n^{\sss{(r)}}=  \left\{ \frac{\log(\vr_n'/\Yrn) + \log\log n}{|\log(\tau-2)|}\right\}.  \ee
Note that $1-a_n^{\sss{(r)}}$ in \eqref{eq:an2} is there to make $t(n^{\vr_n'})$  equal to the upper integer part of the fraction on the rhs of \eqref{eq:an2}. Due to this effect, the last generation has a bit more vertices than $n^{\vr_n'}$, namely
\be \label{eq:mr} Z_{t(n^{\vr_n'})}^{\sss{(r)}}= n^{ \vr_n' (\tau-2)^{a_n^{\sss{(r)}}-1} }=: m_r,\ee
We obtain this expression by rearranging  \eqref{def:yrn-ybn} and using the value $t(n^{\vr_n'})$ from \eqref{eq:an2}.
 The definition of $\vr_n'$ in \eqref{eq:vrn} and $a_n^{\sss{(r)}} \in [0,1)$ implies that the exponent $\vr_n' (\tau-2)^{a_n^{\sss{(r)}}-1}$ is still so small that the condition \eqref{eq:snmax} in Corollary \ref{corr:couple} is satisfied.
Similarly, from \eqref{def:yrn-ybn} and \eqref{eq:an2}, the blue cluster at this moment has size
\be\label{eq:mb}  Z_{t(n^{\vr'})}^{\sss{(b)}}=n^{\vr_n' (\tau-2)^{a_n^{\sss{(r)}}-1} \Ybn/\Yrn}=:m_b,\ee
where the assumption that red reaches size $n^{\vr_n'}$ first is equivalent to the assumption that $\Ybn/\Yrn\le 1$. This assumption together with \eqref{eq:mr} as well as the double-exponential growth apparent from \eqref{def:Y} ensures that the total size of the two BPs is less than $n^{\vr_n}$ and thus Corolllary \ref{corr:couple} still holds. Note that $m_r, m_b$ are \emph{random variables} that are measurable wrt $\CF_{\vr_n'}$.


\subsection{Short path to the hubs through shells}\label{s:shells}
To provide an upper bound on the distance of $v_q$ and the hubs for $q\in\{r,b\}$, as well as to show \eqref{eq:degree-hub}, we build a path from $\CC^{\sss{(q)}}_{t(n^{\vr_n'})}$ to the hubs.
Let us set $C_2:=\max\{C_1, C^\star_1\}$, where $C_1, C^\star_1$ are the constants in the exponent in \eqref{eq:L} for $L_n, L_n^\star$, respectively, and  define the function
\be\label{def:h} h(x):=\exp\left\{ \frac{2C_2}{(\tau-2)^\gamma} (\log x)^\gamma\right\}.\ee
We shall repeatedly use that for any possible $L_n, L_n^\star$ satisfying \eqref{eq:L}, as $x\to \infty$,
\begin{align}\label{eq:h-L1} \min\{ L_n(x^{1/(\tau-2)}), L_n^\star(x^{1/(\tau-2)})\}h(x)&\to \infty, \\
\max\{ L_n(x^{1/(\tau-2)}), L_n^\star(x^{1/(\tau-2)})\}/h(x)&\to 0.\label{eq:h-L2}\end{align}
Recall $m_q$ from \eqref{eq:mr}, \eqref{eq:mb} and that they are measurable wrt $\CF_{\vr_n'}$.  Generally in the rest of the paper, random variables measurable wrt $\CF_{\vr_n'}$ are denoted by small letters, since they can be treated as constants under the measure $\Pv_Y$ in \eqref{eq:cond-measure}, and this is the meaure we mostly work with. In order to build the path, for both $q\in\{r,b\}$ we decompose the high-degree vertices in the graph into the following sets, that we call \emph{shells}:
\be\label{def:Gamma_i} \Gamma_i^{\sss{(q)}}:=\{ v: d_v>\uiq \}, \ee
where $u_i^{\sss{(q)}}$ is defined recursively by
\be\label{eq:ui_recursion} u_{i+1}^{\sss{(q)}} =\left(\frac{\uiq}{h(\uiq)}\right)^{1/(\tau-2)}, \ \ \qquad u_0^{\sss{(q)}}:= \bigg(\frac{m_q}{h(m_q)}\bigg)^{1/(\tau-2)}. \ee
Setting $u_{-1}^{\sss{(q)}}:=m_q$, we obtain by iteration
\[ \uiq= m_q^{(\tau-2)^{-(i+1)}} /\prod_{k=1}^{i+1} h(u_{i-k}^{\sss{(q)}})^{(\tau-2)^{-k}}.\]
Clearly, $\uiq\le m_q^{(\tau-2)^{-(i+1)}}$. Using this upper bound to estimate the arguments of the function $h$ in the denominator, as well as \eqref{def:h}, the  following lower bound holds with $K_\gamma=1-(\tau-2)^{-(1-\gamma)}$:
\be\label{eq:uiq-lower} \uiq\ge m_q^{ (\tau-2)^{-(i+1)}} \exp\left\{ -\frac{2C_2 }{K_\gamma (\tau-2)^\gamma} (\tau-2)^{-(i+1)} (\log m_q)^\gamma \right\}.
\ee
Since $m_q$ tends to infinity with $n$ (see \eqref{eq:mr}, \eqref{eq:mb}) and $\gamma<1$, the second factor is of smaller order than the first factor for all sufficiently large $n$. This observation together with the upper bound yields that for any fixed $i$,
\be\label{eq:uiq-asymp} \uiq  \sim m_q^{ (\tau-2)^{-(i+1)}}, \ee
in the sense of \eqref{def:sim}. Note that $(\tau-2)^{-1}>1$, thus $\uiq$ is growing and $\Gamma_i^{\sss{(q)}} \supset \Gamma_{i+1}^{\sss{(q)}}$.

To show that the initial stage (coupling to BPs) and the paths through shells has nonzero intersection, we will  use the following claim:
\begin{claim}\label{lem:maxdegree}
Let $X_i, \ i=1, \dots, m$ be i.i.d.\  random variables from distribution $F_n^\star$ or $F^\star$.
Then
\be\label{eq:logwhp} \Pv\bigg(\max_{i\in [m]} X_i < \Big(\frac{ m}{h(m)}\Big)^{1/(\tau-2)} \bigg) \le \exp\left\{ - \exp\big\{ \frac{C^\star_1}{(\tau-2)^\gamma} (\log m)^\gamma \big\}\right\} \to 0. \ee

\end{claim}
\begin{proof}We show it for $F_n^\star$. The proof for $F^\star$ is identical.
Clearly
\[ \ba\Pv\bigg(\max_{i\in [m]} X_i < \Big(\frac{ m}{h(m)}\Big)^{1/(\tau-2)} \bigg)&= F^\star_n\left( \big(\tfrac{m}{h(m)}\big)^{1/(\tau-2)}\right)^{m} \\
&\le \exp\left\{ m\left(1-F_n^\star\big((\tfrac{m}{h(m)})^{1/(\tau-2)} \big)\right)\right\}, \ea\]
and the rest follows by using the function $h$ from \eqref{def:h} as well as the form of $F_n^\star$ from \eqref{eq:size-biased2}, in particular the relation in \eqref{eq:h-L1}.\end{proof}
\begin{proof}[Proof of Proposition \ref{prop:hubs}, upper bound]
By the coupling established in Corollary \ref{corr:couple}, conditioned on the size of the last generation that we denote by $m_q$, the degrees in the last generation of the two BPs are an i.i.d.  (either from $F_n^\star$ or from $F^\star$). Claim \ref{lem:maxdegree}, applied conditionally on $m_q$,  ensures that whp there are vertices with degree at least $u_0^{\sss{(q)}}=(m_q/h(m_q))^{1/\tau-2}$ in the last generation of the two BPs, establishing that $\Pv_{\sss{Y}}$-whp
\be\label{eq:pathstart}\CC_{t(n^{\vr_n'})}^{\sss{(q)}}\cap \Gamma_0^{\sss{(q)}} \neq \emptyset.\ee
The next step is to show that $\Pv_{\sss{Y}}$-whp for all $i$ such that $\Gamma_{i+1}^{\sss{(q)}}\neq \varnothing$,
\be\label{eq:gamma-i-connectivity}  \Gamma_i^{\sss{(q)}} \subset N(\Gamma_{i+1}^{\sss{(q)}}), \ee
where  $N(S)$ stands for the set of vertices that are neighbors of $S$.
\begin{figure}\label{fig:mountain}
\includegraphics[width=0.5\textwidth]{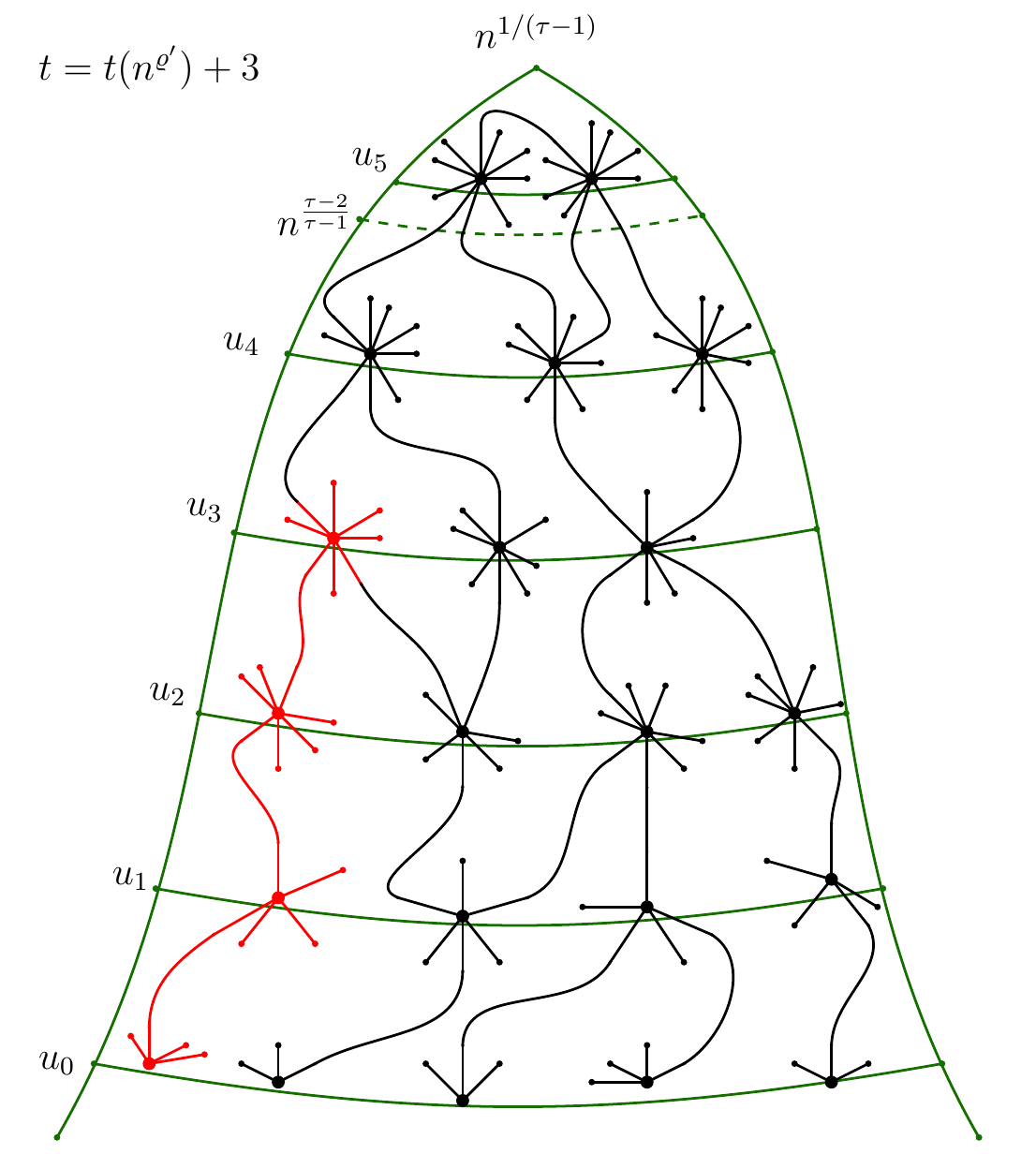}
\caption{An illustration of the layers and the mountain climbing phase at time $t(n^{\vr})+3$. Disclaimer: the degrees on the picture are only an illustration.}
\end{figure}
This statement can be obtained by a modification of \cite[Lemma 3.4]{BarHofKom14} that we provide here for the reader's convenience. Recall that $H(\CA)$ denotes the number of half-edges attached to vertices in a set $\CA$.

The algorithm to generate the configuration model makes it possible that when checking the connection of a vertex $v\in \Gamma_i^{\sss{(q)}}$, we can start by pairing the half-edges of $v$ one after another.  Given $H(\Gamma_{i+1}^{\sss{(q)}})$, the probability that a half-edge is not connected to any of the half-edges attached to vertices in $\Gamma_{i+1}^{\sss{(q)}}$ is at most $1-H(\Gamma_{i+1}^{\sss{(q)}})/\ell_n$. Since we can pair at least $\uiq/2$ half-edges before all the half-edges of $v$ are paired\footnote{This is since in the worst case scenario the first $\uiq/2$ half-edges are paired all back to half-edges of the same vertex $v$.}, by a union bound for all $v\in H(\Gamma_{i}^{\sss{(q)}})$, 
\be\ba\label{eq:no-connection}
\Pv_{\sss{Y}}\big(\exists v\in \Gamma_i^{\sss{(q)}}: v \not \leftrightarrow \Gamma_{i+1}^{\sss{(q)}}\mid H(\Gamma_{i+1}^{\sss{(q)}} ) \big) &\le |\Gamma_i^{\sss{(q)}}| \Big(1-\frac{H(\Gamma_{i+1}^{\sss{(q)}})}{\ell_n} \Big)^{\uiq/2}\\
&\le n \exp\!\left\{  -C\uiq \uiqp (1-F_n(\uiqp)) \right\}\!.
\ea \ee
By using the lower bound in \eqref{eq:L} as well as the upper bound $\uiqp\le (\uiq)^{1/(\tau-2)}$, (see also \eqref{eq:h-L1}),
\[ \uiq \uiqp (1-F_n(\uiqp))= h(\uiq)L_n(\uiqp)\ge \exp\left\{ \wit C (\tau-2)^{-i} (\log m_q)^\gamma\right\},  \]
for some $\wit C>0 $. Since $m_q$ (\eqref{eq:mr}, \eqref{eq:mb}) tends to infinity with $n$, the bound in \eqref{eq:no-connection} tends to zero as $n\to \infty$ even when we sum over $i\ge 1$.
 This establishes \eqref{eq:gamma-i-connectivity}.

 The coupling to two BPs combined with \eqref{eq:pathstart} and \eqref{eq:gamma-i-connectivity} establishes the existence of a path to the hubs. This provides an \emph{upper bound} on the distance between $v_r, v_b$ and the hubs.
It remains to calculate the length of these paths.

We write $i_{\star\sss{(q)}}$ for the last index when $\Gamma_{i}^{\sss{(q)}}$ is nonempty, i.e., by \eqref{eq:F},
\be\label{def:i*}i_{\star\sss{(q)}}:=\sup \{ i: u_i^{\sss{(q)}} \le n^{{\beta_n}} \}. \ee
Some calculation using the value of $m_q$ from \eqref{eq:mr}, combined with \eqref{eq:uiq-asymp} and \eqref{eq:uiq-lower} shows that
\be\label{eq:value_i*-al} i_{\star \sss{(r)}}= -1+ \frac{\log ({\beta_n}/(\vr_n' (\tau-2)^{a_n^{\sss{(r)}}-1}))}{|\log(\tau-2)|}-\bnr({\beta_n}) + o_{\Pv_{\sss{Y}}}(1), \ee
where $\bnr({\beta_n})$ is the fractional part of the previous term on the rhs.
Using the value of $a_n ^{\sss{(q)}}$ from \eqref{eq:an}, plus the fact that $\{x - 1+\{y\}\}=\{x+y\}$,  we get that
\be\label{eq:bnr} \bnr({\beta_n}) = \left\{ \frac{\log ({\beta_n}/\vr_n')}{|\log(\tau-2)|}+ a_n^{\sss{(r)}}-1\right\} = \left\{ \frac{\log ({\beta_n}/\Yrn)+\log \log n}{|\log(\tau-2)|}\right\},\ee
exactly as defined in \eqref{def:ti-bi}. Similar calculation for $q=b$ yields
\be\label{eq:istar-b} i_{\star \sss{(b)}} =-1+\frac{\log ({\beta_n}/ \vr_n' (\tau-2)^{a_n^{\sss{(r)}}-1})) + \log (\Yrn/\Ybn)}{|\log(\tau-2)|}-\bnb({\beta_n})+o_{\Pv_{\sss{Y}}}(1),\ee
where $\bnb({\beta_n})$ is the fractional part of the previous term on the rhs. Again, some calculation yields that $\bnb({\beta_n})$ is exactly as in \eqref{def:ti-bi}. Note that the definition of $\vr_n'$ in \eqref{eq:vrn} guarantees that the ratio ${\beta_n}/\vr_n'$ is bounded, and hence $i_{\star\sss{(q)}}$ is a tight random variable (measurable wrt $\CF_{\vr_n'}$).
From \eqref{eq:uiq-asymp} and \eqref{eq:value_i*-al} and \eqref{eq:istar-b} respectively, one can calculate that
\be\label{eq:ui*}  u_{i_{\star\sss{(q)}}}^{\sss{(q)}}\sim n^{{\beta_n}(\tau-2)^{\bnq({\beta_n})}}, \ee
and the error factor as in \eqref{eq:uiq-lower} is $o_{\Pv_{\sss{Y}}}(1)$, since $i_{\star\sss{(q)}}$ does not tend to infinity with $n$. Thus, the total length of the constructed path is, for $q\in\{r,b\}$,
\be\label{eq:k*+i*}  T_q({\beta_n})=t(n^{\vr_n'})+i_{\star\sss{(q)}}=\frac{\log\log (n^{{\beta_n}})-\log Y_n^{\sss{(q)}}}{|\log(\tau-2)|} -1-\bnq({\beta_n}),\ee
establishing the upper bound on $\mathrm d_G(v_q, \mathrm{hubs})$ in Proposition \ref{prop:hubs}.
Note that $T_q({\beta_n})$ only depends  on the value $\vr_n'$ through the approximating variables $Y_n^{\sss{(q)}}$, and $\bnq({\beta_n})$ is exactly the fractional part of the expression on the rhs of $T_q({\beta_n})$. Since also $Y_n^{\sss{(q)}}\toindis Y_q$ irrespective of the choice of $\vr_n'$, this establishes that the choice of $\vr_n'$ is not relevant in the proof (at least not in the limit), but more a technical necessity.
\end{proof}
\subsection{Upper bound on the degrees in the BFS}\label{s:badpaths}
Now we turn towards providing a  \emph{matching lower bound} for the distance from the hubs.
Similarly as in \eqref{eq:ui_recursion}, let us define, for $q\in\{r,b\}$,
\be\ba\label{eq:uibar}
   \widehat u_0^{\sss(q)}&:= (m_q h(m_q))^{1/(\tau-2)}, \\
  \widehat u_{i+1}^{\sss{(q)}}&:= \big(\widehat u_i^{\sss{(q)}}h(\huiq)\big)^{1/(\tau-2)},  \\
\widehat\Gamma_i^{\sss{(q)}}&:=\{v \in \CMD: d_v \ge \widehat u_i^{\sss{(q)}}\},
\ea\ee
Note that $\huiq$ grows faster than $\uiq$ since here we multiply by $h$ instead of dividing by it.

The next lemma handles the upper bound on the maximal-degree vertex reached at any time $t(n^{\vr_n'})+i$, but first some definitions. We say that a sequence of vertices and half-edges $\underline \pi:=(\pi_0, s_0, t_1, \pi_1, s_1, t_2,  \dots,  t_k, \pi_k)$ forms a \emph{path} in $\CMD$, if for all $0< i\le k$, the half-edges $s_i, t_i$ are incident to the vertex $\pi_i$ and $(s_{i-1}, t_i)$ forms an edge between $\pi_{i-1},\pi_i$.

For $q\in\{r,b\}$, we say that a path is \emph{q-good} if $\deg(\pi_i)\le\widehat u_i^{\sss{(q)}}$ holds for every $i$. Otherwise we call it \emph{$q$-bad}. We further decompose the set of $q$-bad paths in terms of where they `turn' $q$-bad:
\be\label{def:badpaths} \ba   Bad\CP_k^{\sss{(q)}} := &\{ (\pi_0, s_0, t_1, \pi_1, s_1 \dots, t_k, \pi_k) \text{ is a path, } \\
 &\quad \pi_0\!\in\!\CC^{\sss{(q)}}_{ t(n^{\vr_n'}) },\   \deg(\pi_i)\!\le\! \widehat u_i^{\sss{(q)}} \ \forall i\le  k-1,\ \deg(\pi_k)\!>\!\widehat u_k^{\sss{(q)}}   \}.\ea\ee
The following lemma shows that $q$-bad paths $\Pv_{\sss{Y}}$-whp do not occur:
\begin{lemma}\label{lem:badpaths}
For some constant $C>0$, the following bound on the probability of having any bad paths holds for color $q\in\{r,b\}$:
\be\label{eq:badpath} \Pv_{\sss{Y}}( \exists k\in [0, i_{\star\sss{(q)}}]:  Bad\CP_k^{\sss{(q)}} \neq \varnothing) \le C \exp\{-C (\log m_q)^{\gamma}\}.\ee
\end{lemma}
Before we prove this lemma, we need the following technical claim:
\begin{claim}\label{cl:tech}
Let $(D_{n,i}^\star)_{i\le m}$ be i.i.d.\ from distribution $F_n^\star$ or $F^\star$.  Then there exists a $0<C<\infty$, so that
\be\label{eq:iid-sum} \Pv_{\sss{Y}}\left(  \sum_{i=1}^{m} D_{n,i}^\star \ge (m h(m))^{1/(\tau-2)} \right)  \le \exp\{ -C (\log m)^\gamma\}. \ee
Further, for any $y\in [0, n^{{\beta_n}}]$,
\be\label{eq:upper-mom} \sum_{d_{\pi}\ge y} \frac{d_{\pi}  }{\ell_n} \le L_n^{\star, \text{up}}(y)y^{2-\tau} \ee
where we denote by $L_n^{\star, \text{up}}$ the upper bound in \eqref{eq:L} on $L_n^\star$. Next, the empirical truncated second moment satisfies for all $y_n\to \infty$ and large enough $n$ that
\be \label{eq:second-mom}\sum_{\pi: d_{\pi}\le y_n} \frac{d_{\pi} (d_{\pi}-1)}{\ell_n} \le \frac{2}{3-\tau}(y_n)^{3-\tau} L_n^{\star, \text{up}}(y_n).\ee
\end{claim}
\begin{proof}
The proof of \eqref{eq:upper-mom} is the probably the easiest.
Namely, by the definition of the empirical distribution as well as $F_n^\star$ the sum can be rewritten as follows:
\be \sum_{d_{\pi}\ge y} \frac{d_{\pi}  }{\ell_n}= \sum_{j\ge y} \frac{\sum_{v\in [n]} j \ind_{\{d_v=j\}}}{\ell_n}= 1-F_n^\star(y-1), \ee
and an application of \eqref{eq:size-biased2} establishes \eqref{eq:upper-mom}. Next, \eqref{eq:second-mom} can be rewritten similarly,
\be\label{eq:second-trunc-mom} \ba \sum_{\pi: d_{\pi}\le y_n} \frac{d_{\pi} (d_{\pi}-1)}{\ell_n} &=  \sum_{j\le y_n} (j-1) \frac{ \sum_{v\in [n]} j\ind_{\{d_v=j\}}}{\ell_n}=\sum_{j\le y_n} (j-1) \Pv(D_n^\star=j-1)\\
& \le \sum_{j\le y_n} (1-F_n^\star(j)) \le \sum_{j\le y_n} \frac{L_n^{\star}(j)}{j^{\tau-2}}\le \sum_{j\le y_n} \frac{L_n^{\star, \text{up}}(j)}{j^{\tau-2}}.\ea  \ee
To obtain the second line, we used the usual trick to relate the expectation to the tail of a distribution, namely,
\[ \ba\sum_{j\le y_n} (j-1) \Pv(D_n^\star=j-1)&=\sum_{j\le y_n} \sum_{s\le j-1} \Pv(D_n^\star=j-1) = \sum_{s\le y_n-1} \sum_{s< j \le y_n} \Pv(D_n^\star=j-1)\\
&\le \sum_{s\le y_n-1} \Pv(D_n^\star \ge s) = \sum_{s\le y_n} (1-F_n^\star(s)).\ea\]
The condition that  $y_n\to \infty$  as $n\to \infty$ enables us to apply the direct half of Karamata's theorem\footnote{Here we use that the \emph{upper bound} $L_n^{\star, \text{up}}$ on the function $L_n^\star$ is slowly varying.} (see \cite[Page 26]{BinGol}), and obtain that for all large enough $n$, the following bound holds on the rhs of \eqref{eq:second-trunc-mom}:
\be\label{eq:karam} \sum_{j\le y_n} \frac{L_n^{\star, \text{up}}(j)}{j^{\tau-2}}\le \frac{2}{3-\tau}(y_n)^{3-\tau} L_n^{\star, \text{up}}(y_n),\ee
finishing the proof of \eqref{eq:second-mom}.
The proof of \eqref{eq:iid-sum} is the trickiest, we handle it with a truncation method. Let us shortly write $M_m:=(m h(m))^{1/(\tau-2)}$.
First we use a union bound:
\be\ba\label{eq:single-big-union}\Pv_{\sss{Y}}\left(  \sum_{i=1}^{m} D_{n,i}^\star \ge M_m \right) &\le \Pv\left( \exists i\le m:  D_{n,i}^\star \ge M_m\right) \\
&\ + \Pv\left(  \sum_{i=1}^{m} D_{n,i}^\star\ind_{\{D_{n,i}^\star <  M_m \}} \ge M_m \right). \ea\ee
Then, we
estimate the first term on the rhs of \eqref{eq:single-big-union} again by a union bound:
\be\label{eq:single-big} \Pv\left( \exists i\le m:  D_{n,i}^\star \ge M_m \right) \le  m (1-F_n^{\star}(M_m))=L_n^\star(M_m) / h(m).\ee
We can use Markov's inequality on the second term on the rhs of \eqref{eq:single-big-union}:
\be\label{eq:truncated-markov} \Pv_{\sss{Y}}\left(  \sum_{i=1}^{m} D_{n,i}^\star\ind_{\{D_{n,i}^\star < M_m \}}  \ge M_m \right) \le \frac{m \Ev[D_{n,i}^\star\ind_{\{D_{n,i}^\star < M_m \}}]}{M_m } \le C\frac{m M_m^{3-\tau} L_n^\star(M_m)  }{M_m },
 \ee
 where we have used that the expectation in the numerator equals precisely the truncated empirical second moment as in \eqref{eq:second-mom} with $y_n:=M_m$, thus this expectation can be handled in the same way as the rhs of \eqref{eq:second-trunc-mom}. After elementary calculation, the sum of rhs of \eqref{eq:single-big} and \eqref{eq:truncated-markov} equals
\be\label{eq:first-error-term} (C+1) \frac{L_n^\star(M_m) }{h(m)}\le (C+1)\exp\left\{ -\frac{C_2}{(\tau-2)^\gamma} (\log m)^\gamma\right\}. \ee This, together with \eqref{eq:single-big-union} establishes \eqref{eq:iid-sum}.
\end{proof}

\begin{proof}[Proof of Lemma \ref{lem:badpaths}]
The statement can be proved by path-counting methods in the same way as \cite[Lemma 5.2]{BarHofKom14} in the Appendix of that paper.
Some minor modifications are needed to that proof, for two reasons: First, Assumption \ref{assu:degree-dist} imposes an assumption on the \emph{empirical degree distribution} $F_n$ and the degrees are no longer i.i.d.. This makes certain estimates about truncated empirical moments easier (i.e., \eqref{eq:upper-mom} and \eqref{eq:second-mom}). On the other hand, Assumption \ref{assu:degree-dist} is weaker than
the assumption on the degrees there, compare to \cite[(1.1)]{BarHofKom14}. As a result, the recursion of $\huiq$ uses the function $h$ as a multiplier instead of the constant function $C\log n$. To make the argument easy to follow here, we recall the main steps in the proof and highlight differences only.

First of all, formulas \cite[(A.3)-(A.5)]{BarHofKom14} that count the expected number of bad paths apply word-for-word. That is,
the expected number of paths in $\CMD$ through a fixed sequence of vertices $(\pi_0, \pi_1, \dots, \pi_k)$ equals
\be\label{eq:combi} \prod_{i=1}^k \frac{1}{\ell_n^\star-2i+1} d_{\pi_0} \left(\prod_{i=1}^{k-1} d_{\pi_i} (d_{\pi_i}-1)\right) d_{\pi_k}, \ee
where $\ell_n^\star$ denotes the number of unpaired half-edges at the moment of counting these paths. It is not hard to see (using the sizes $m_q$ in  \eqref{eq:mr}\eqref{eq:mb}) that $\ell_n^\star=\ell_n (1-o_{\Pv}(1))$ when we apply this to a path emanating from $\CC^{\sss{(q)}}_{t(n^{\vr_n'})}$.

This formula holds for any fixed sequence of vertices. We can count the expected number of $q$-bad paths in $Bad\CP_k^{\sss{(q)}}$ when we impose the same restrictions on  $\pi_i$ as in \eqref{def:badpaths}, and sum over all possible such options:
\be\label{eq:A5} \Ev_{\sss{Y}}\left[ Bad\CP_k^{\sss{(q)}}|\right] \le \e^{Ck^2/\ell_n}\sum_{\pi_0 \in \CC^{\sss{(q)}}_{t(n^{\vr_n'})} }\!\!\!\!d_{\pi_0}   \prod_{i=1}^{k-1}\Bigg(\sum_{\pi_i: d_{\pi_i}\le \widehat u_i^{\sss{(b)}}}  \frac{d_{\pi_i} (d_{\pi_i}-1)}{\ell_n^\star}\Bigg) \Bigg(\sum_{{\substack{ \pi_k \in [n] \\ d_{\pi_k} \ge \widehat u_k^{\sss{(b)}}}}}\frac{d_{\pi_k}}{\ell_n^\star}\Bigg).\ee

The next step is to estimate the different factors on the rhs of \eqref{eq:A5}: here, the two proofs separate. In fact, we can use the bounds in Claim \ref{cl:tech}, and then we can spare all the arguments between \cite[(A.6)-(A.9)]{BarHofKom14}.
Let $\CE_n:=\{ \sum_{i=1}^{m_q} D_i^\star \le \widehat u_0^{\sss{(q)}} \}$, so that $\CE_n$ holds $\Pv_{\sss{Y}}$-whp by Claim \ref{cl:tech}, with $m:=m_q$.
 Using the estimates \eqref{eq:upper-mom} and \eqref{eq:second-mom} in Claim \ref{cl:tech} with $y_n=\huiq$,  \eqref{eq:A5} turns into
\be\label{eq:A11} \Ev_{\sss{Y}}[ |Bad\CP_k^{\sss{(q)}}| \mid \CE_n] \le \widehat u_0^{\sss{(q)}} \cdot   (\widehat u_k^{\sss{(q)}})^{2-\tau} L_n^{\star, \text{up}}(\widehat u_{k}^{\sss{(q)}})  \cdot \prod_{i=1}^{k-1} (\huiq)^{3-\tau} \frac{2}{3-\tau}L_n^{\star, \text{up}} (\huiq).\ee
This formula replaces \cite[(A.11)]{BarHofKom14}.
It is an elementary calculation using the defining recursion \eqref{eq:uibar} that $\widehat u_j^{\sss{(q)}} \cdot (\widehat u_{j+1}^{\sss{(q)}})^{3-\tau}h(\widehat u_j^{\sss{(q)}})=\widehat u_{j+1}^{\sss{(q)}}$.
Applying this equation to $j=0, \dots, k-2$ sequentially, we arrive at the identity
\be\label{eq:identity}  \widehat u_0^{\sss{(q)}} \cdot   (\widehat u_k^{\sss{(q)}})^{2-\tau} h(\widehat u_{k-1}^{\sss{(q)}})  \cdot \prod_{i=1}^{k-1} (\huiq)^{3-\tau} h (\widehat u_{i-1}^{\sss{(q)}}) =1.\ee
Comparing \eqref{eq:A11} to \eqref{eq:identity},
we see that\footnote{Let us compare the inequality \eqref{eq:A11-bound} with the bound $O((\log n)^{-k})$ on $\Ev_{\sss{Y}}[ |\CC^{\sss{(b)}} ad\CP_k^{\sss{(q)}}|$ in \cite{BarHofKom14} that can be found in the second formula after \cite[(A.11)]{BarHofKom14}. In \cite{BarHofKom14} the recursion in \eqref{eq:uibar} was used with the special choice $h(x)\equiv C\log n$,  and $L_n^\star$ was assumed to be a bounded function. Using these, the bound $O((\log n)^{-k})$ is a special case of the bound here in \eqref{eq:A11-bound}. Thus, \eqref{eq:A11-bound} is a generalisation of the bound in \cite{BarHofKom14} for general choice of $h$ in the  recursion \eqref{eq:uibar}. The generalisation was necessary since in this paper we allow a wide range of $L_n$ in \eqref{eq:F} while in \cite{BarHofKom14} the more restrictive $0<c\le L_n\le C<\infty$ assumption was set.}
\be\label{eq:A11-bound} \Ev_{\sss{Y}}\left[ |Bad\CP_k^{\sss{(q)}}| \mid \CE_n \right] \le \prod_{i=0}^{k-1} \frac{2}{3-\tau} \frac{L_n^{\star, \text{up}} (\widehat u_{i+1}^{\sss{(q}})}{h(\huiq)}.  \ee
Next we show that for all $n$ large enough, the rhs of \eqref{eq:A11-bound} tends to zero even when summed over all $k\ge 1$. For this, using the recursion on $\huiq$ in \eqref{eq:uibar} as well as the function $h$ in \eqref{def:h},
\be\label{eq:elln} \ba L_n^{\star, \text{up}}(\widehat u_{i+1}^{\sss{(q}}) &= \exp\left\{ C^\star_1 \left(\log (\huiq)^{1/(\tau-2)} + 2 C_2 (\log (\huiq)^{1/(\tau-2)})^\gamma \right)^\gamma  \right\}\\
&=\exp\left\{ C^\star_1 \left(\log (\huiq)^{1/(\tau-2)}\right)^\gamma \left(1+ 2 C_2 (\log (\huiq)^{1/(\tau-2)})^{\gamma-1} \right)^\gamma \right\}  \\
&\le\exp\left\{ \frac{3C^\star_1}{2} \left(\log (\huiq)^{1/(\tau-2)}\right)^\gamma  \right\},  \ea\ee
where in the inequality we have used that since $\gamma<1$ and $\huiq$ tends to infinity with $n$, for all $n$ large enough the last factor in the rhs of the second line is at most $3/2$.
The denominator of the $i$th factor in \eqref{eq:A11-bound} has the exact same form, except there the constant multiplier in the exponent is at least $2C_1^\star$. Thus, the rhs of \eqref{eq:A11} is at most
\be\label{eq:A11-bound-2} \ba\Ev_{\sss{Y}}\left[ |Bad\CP_k^{\sss{(q)}}| \mid \CE_n \right] &\le \prod_{i=0}^{k-1} \frac{2}{3-\tau} \exp\left\{ -\frac{C^\star_1}{2} \left(\log ((\huiq)^{1/(\tau-2)})\right)^\gamma  \right\}\\
&\le \frac{2^k}{(3-\tau)^k} \exp\left\{  -\frac{C^\star_1}{2}\sum_{i=0}^{k-1} (\tau-2)^{-(i+2)\gamma } (\log m_q)^\gamma\right\}, \ea \ee
where we used the lower bound $\huiq\ge m_q^{(\tau-2)^{-(i+1)}}$ that follows from the recursion \eqref{eq:uibar} (see also \eqref{eq:wideui-ui} below) to obtain the second line.
Since $(\tau-2)^{-\gamma}>1$, the sum in the exponent is of order $(\tau-2)^{-k\gamma}(\log m_q)^\gamma$. So, by Markov's inequality we obtain for some constant $C>0$ that
\be\label{eq:no-bad-path}\Pv_{\sss{Y}}\!\left(\exists k\ge 1,   Bad\CP_k^{\sss{(q)}} \neq \varnothing \right)\le \sum_{k=1}^\infty  \Ev_{\sss{Y}}[| Bad\CP_k^{\sss{(q)}}|] \le C \exp\{- C (\log m_q)^\gamma   \}\to 0 \ee
as $n \to \infty$, since $m_q$ is a positive power of $n$ under $\Pv_{\sss{Y}}$ (see \eqref{eq:mr}, \eqref{eq:mb}).
We yet have to add the case $k=0$: note that $Bad\CP_0^{\sss{(q)}} \neq \varnothing$ means that there exists a vertex in the last generation of the BP  with degree at least $\widehat u_{0}^{\sss{(q)}}$. We have already estimated this probability in \eqref{eq:single-big}, and the error term obtained in \eqref{eq:first-error-term} can be merged into the rhs of \eqref{eq:no-bad-path}, establishing the statement of the lemma in \eqref{eq:badpath}.
\end{proof}
\begin{proof}[Proof of Proposition \ref{prop:hubs}, lower bound] We argue that $T_q({\beta_n})$ is also whp a lower bound to reach the hubs, that is, there is whp no path to the hubs shorter than $T_q({\beta_n})$.

On the event $\{ \forall k\in [0,  i_{\star \sss{(q)}}]: Bad\CP_k^{\sss{(q)}} = \varnothing, q\in\{r,b\} \}$, that occurs $\Pv_{\sss{Y}}$-whp by Lemma \ref{lem:badpaths}, we can use the upper bound $\widehat u_i^{\sss{(q)}}$ on the degrees at time $t(n^{\vr_n'})+i$ for all $i\le i_{\star \sss{(q)}}$. Hence
we obtain that the time it takes to reach a vertex of  degree at least $n^{(\tau-2)/(\tau-1)}$ is at least
\[ \widehat i_{\star \sss{(q)}} :=\inf \{ i: \huiq \ge n^{{\beta_n} (\tau-2)} \},
\]
which, considering the double exponential growth of $\widehat u_i^{\sss{(q)}}$ by powers of $1/(\tau-2)$, is similar to the definition of  $ i_{\star \sss{(q)}}$ in \eqref{def:i*}.
The lower bound follows once we show that $\widehat i_{\star \sss{(q)}}=i_{\star \sss{(q)}} $ holds $\Pv_{\sss{Y}}$-whp. For this it is enough to show that \eqref{eq:uiq-asymp} holds also for $\huiq$.
From the recursion \eqref{eq:uibar},
\be\label{eq:wideui-ui} \widehat u_i^{\sss{(q)}} = m_q^{(\tau-2)^{-(i+1)}}
\prod_{k=1}^{i+1}h(\widehat u_{i-k}^{\sss{(q)}})^{(\tau-2)^{-k}}. \ee
After a somewhat lengthy calculation, using a similar argument as in the second and third line of \eqref{eq:elln} recursively, we obtain that the product on the rhs is at most
\[\begin{split}  \exp\left\{  2 C_2(\log (m_q^{1/(\tau-2)}))^\gamma \frac{1}{(\tau-2)^{i+1}}\sum_{k=1}^{i} \frac{1}{(\tau-2)^{k(\gamma-1)}} \right\} \\ \le \exp\left\{  \wit C  (\log (m_q^{1/(\tau-2)}))^\gamma  \frac{1}{(\tau-2)^{i+1}} \right\},\end{split}\]
since $\gamma<1$. Recall that $m_q$ tends to infinity with $n$ and  comparing this to  \eqref{eq:wideui-ui} as well as to the definition of $\sim$ in \eqref{def:sim}.
Nevertheless, the product in \eqref{eq:wideui-ui} is of much smaller order than the main term $m_q^{(\tau-2)^{-(i+1)}}$ and thus
we obtain that \eqref{eq:uiq-asymp} holds and also that $\widehat i_{\star\sss{(q)}}=i_{\star \sss{(q)}} $ $\Pv_{\sss{Y}}$-whp. This finishes the proof of the lower bound.
\end{proof}
\section{Early meeting is unlikely}\label{s:no-early}
For the lower bound of Theorem \ref{thm:distances}, we crucially use the following proposition that shows that the two explorations are disjoint,  i.e., the vertices at distance at most $T_b({\beta_n})$ away from $v_b$ are all different from the vertices that are distance at most $T_r({\beta_n})$ away from $v_r$:
\begin{proposition}\label{lem:no-early-meeting} 
Let us consider the configuration model on $n$ vertices with empirical degree distribution that satisfies Assumption \ref{assu:degree-dist}, and let $v_r, v_b$ be two uniformly chosen vertices.
The event
\be \label{eq:no-early-meeting} \CC^{\sss{(r)}}_{T_r({\beta_n})}\cap \CC^{\sss{(b)}}_{T_b({\beta_n})} = \varnothing \ee
holds $\Pv_{\sss{Y}}$-whp.
Further, the total number of half-edges attached to vertices in $\CC^{\sss{(r)}}_{T_r({\beta_n})}, \CC^{\sss{(b)}}_{T_b({\beta_n})}$ is the same order of magnitude as the degree of $v_q^\star$ in Proposition \ref{prop:hubs} up to smaller order correction terms.
That is, $\Pv_{\sss{Y}}$-whp,
\be \label{eq:half-edge-red-bound} H(\CC^{\sss{(r)}}_{T_r({\beta_n})}) \sim n^{{\beta_n}(\tau-2)^{\bnr({\beta_n})} } \quad \mbox{and}\quad   H(\CC^{\sss{(b)}}_{T_b({\beta_n})}) \sim n^{{\beta_n}(\tau-2)^{\bnb({\beta_n})} }.
\ee
\end{proposition}
\begin{proof}
Recall that we write $\Pv_{\sss{Y}}(\cdot), \Ev_{\sss{Y}}[\cdot]$ for probabilities of events and expectations of random variables conditioned on $\CF_{\vr_n'}$. Recall from Lemma \ref{lem:badpaths} that the event $\text{NoBad}:=\{ Bad \CP^{\sss{(q)}}_k =\varnothing \  \forall k\le i_{\star \sss{(q)}} \mbox{ for } q\in\{r,b\}\} $ holds $\Pv_Y$-whp. Since for any event $A$, $\Pv_Y(A)\ge \Pv_Y(A\mid \text{NoBad}) \Pv_Y(\text{NoBad})$, it is enough to show that the event in \eqref{eq:no-early-meeting} holds with probability tending to $1$ when conditioned on $\text{NoBad}$.

To prove the proposition  we first calculate the total number of free (unpaired) half-edges going out of the set $\CC^{\sss{(r)}}_{T_q({\beta_n})-\ell}$, (that we denote by $H(\CC^{\sss{(r)}}_{T_q({\beta_n})-\ell})$), for any $\ell\in [0, i_{\star \sss{(q)}}], q\in\{r,b\}$. We do this by counting the number of paths \emph{with free ends}:  we say that a sequence of vertices and half-edges $(\pi_0, s_0, t_1, \pi_1, s_1, t_2,  \dots,  t_k, \pi_k, s_k)$ forms a  \emph{free-ended path} of length $k$ in $\CMD$, if for all $0< i\le k$, the half-edges $s_i, t_i$ are incident to the vertex $\pi_i$ and $(s_{i-1}, t_i)$ forms an edge between vertices $\pi_{i-1},\pi_i$. Clearly, since the same vertex might be approached on several paths, the total number of free half-edges in $\CC^{\sss{(q)}}_{T_q({\beta_n})-\ell}$ can be bounded from above by the number of free-ended paths of length $T_q({\beta_n})-\ell$, starting from $v_r$. By the definition of $ Bad \CP^{\sss{(q)}}_k$ in \eqref{def:badpaths}, on the event $\text{NoBad}$ at time $t(n^{\vr_n'}) + i$, $\widehat u_{i}^{\sss{(q)}}$ defined in \eqref{eq:uibar}  is an upper bound on the degrees of color $q$ vertices.
We write $\CN_k(\CA, \text{free})$ for the set of, and $N_k(\CA, \text{free})$ for the total number of, $k$-length free-ended paths starting from an unpaired half-edge that belongs to the set $\CA$. Then, since $T_q({\beta_n})=t(n^{\vr_n'})+i_{\star \sss{(q)}}$ (see \eqref{eq:k*+i*}), for any $\ell\le i_{\star \sss{(q)}}$,
\be\label{eq:half-edge-to-path}H(\CC^{\sss{(r)}}_{T_q({\beta_n})-\ell}) \le  N_{i_{\star \sss{(q)}}-\ell}(\CC^{\sss{(r)}}_{t(n^{\vr_n'})}, \text{free}), \ee
and recall that $\CC^{\sss{(r)}}_{t(n^{\vr_n'})}$ is coupled to the branching process described in Section \ref{s:couple}. Hence, the degrees in the last generation of the BP phase are i.i.d.\ having distribution $D^\star_n$ satisfying \eqref{eq:size-biased2}.
When counting free-ended paths through fixed vertices $(\pi_0, \dots, \pi_k)$, \eqref{eq:combi} should be modified so that we have to choose two half-edges also from the end vertex $\pi_k$, thus there is an additional factor $d_{\pi_k}-1$ that should multiply \eqref{eq:combi}.  The effect of this on \eqref{eq:A5} is that the factor containing $\pi_k$ can be merged into the previous factor:
\be\ba\label{eq:nkab} \Ev_{\sss{Y}}&\left[  N_{i_{\star \sss{(q)}}-\ell}(\CC^{\sss{(r)}}_{t(n^{\vr})}, \text{free})\mid \text{NoBad} \right] \\ &\quad \quad \quad\le \e^{Ci_{\star \sss{(q)}}^2 /\ell_n}\!\!\!\sum_{\pi_0 \in \CC^{\sss{(q)}}_{t(n^{\vr_n'})} } \!\!\!\!d_{\pi_0} \cdot\prod_{i=1}^{i_{\star \sss{(q)}}-\ell    }\left(\sum_{\pi_i \in \Lambda_i} \frac{d_{\pi_i} (d_{\pi_i}-1)}{\ell_n}\right), \ea\ee
where we have applied the restriction that is valid under the event $\text{NoBad}$: $\pi_i\in \Lambda_i$, with $\Lambda_i=\{v \in [n]: D_v \le  \widehat u_i^{\sss{(r)}} \}$. Note that we could use again that $\ell_n^\star=\ell_n (1-o_{\Pv}(1))$ by the same argument that was used after formula \eqref{eq:combi}.
Using  \eqref{eq:iid-sum} and \eqref{eq:second-mom} from Claim \ref{cl:tech}, we obtain that
 \be\label{eq:free-k-1} \Ev_{\sss{Y}}[  N_{i_{\star \sss{(r)}}-\ell}(\CC^{\sss{(r)}}_{t(n^{\vr})}, \text{free})\mid \text{NoBad}] \le \widehat u_0^{\sss{(r)}} \cdot \left( \prod_{i=1}^{i_{\star \sss{(r)} }-\ell } \frac{2}{3-\tau}(\widehat u_i^{\sss{(q)}})^{3-\tau}  L_n^{\star \text{up}}(\widehat u_i^{\sss{(q)}}) \right) \e^{2 i_{\star \sss{(q)}}^2/ \ell_n }. \ee
Note that this is similar to \eqref{eq:A11}. Indeed, we again sequentially apply the identity $\widehat u_j^{\sss{(q)}} \cdot (\widehat u_{j+1}^{\sss{(q)}})^{3-\tau}h(\widehat u_j^{\sss{(q)}})=\widehat u_{j+1}^{\sss{(q)}}$, and then \eqref{eq:free-k-1} turns into
\be\label{eq:free-k-2} \Ev_{\sss{Y}}[ N_{i_{\star \sss{(r)}}}(\CC^{\sss{(r)}}_{t(n^{\vr})}, \text{free}) \mid \text{NoBad} ] \le \widehat u_{   i_{\star \sss{(r)} }-\ell }^{\sss{(q)}} \left(\prod_{i=0}^{ i_{\star \sss{(r)} }-\ell -1 } \frac{2}{3-\tau} \frac{L_n^{\star, \text{up}} (\widehat u_{i+1}^{\sss{(q}})}{h(\huiq)} \right)  \e^{2 i_{\star \sss{(q)}}^2/ \ell_n }. \ee
Combining this with Markov's inequality and a union bound gives
\be\ba\label{eq:half-edge-red} \Pv_{\sss{Y}}&\left(\exists \ell\le i_{\star \sss{(q)}}: H(\CC^{\sss{(r)}}_{T_q({\beta_n})-\ell}) \ge  \widehat u_{ i_{\star \sss{(r)} }-\ell}^{\sss{(q)}}\mid \text{NoBad} \right)  \\  &\qquad \qquad\qquad\le\e^{2 i_{\star \sss{(q)}}^2/ \ell_n } \sum_{\ell=0}^{i_{\star \sss{(q)}}}\left(\prod_{i=0}^{ i_{\star \sss{(r)} }-\ell -1 } \frac{2}{3-\tau} \frac{L_n^{\star, \text{up}} (\widehat u_{i+1}^{\sss{(q}})}{h(\huiq)} \right).     \ea\ee
Recall again that $i_{\star \sss{(q)}}$ is a tight random variable  measurable wrt $\CF_{\vr_n'}$ (see \eqref{eq:value_i*-al} and \eqref{eq:istar-b}), and $\ell_n=\Ev[D_n]n$ is of order $n$. Thus the first factor on the rhs is $1+o_{\Pv_{\sss{Y}}}(1)$. Further, in the analysis below \eqref{eq:A11-bound} we have showed that the sum in the rhs of \eqref{eq:half-edge-red} is at most the rhs of \eqref{eq:no-bad-path}. Thus we obtain
\be\label{eq:half-edge-red-2} \Pv_{\sss{Y}}\left(\exists \ell\le i_{\star \sss{(q)}}: H(\CC^{\sss{(r)}}_{T_q({\beta_n})-\ell})\ge  \widehat u_{ i_{\star \sss{(q)} }-\ell}^{\sss{(q)}}\mid \text{NoBad} \right)  \le C\exp\left\{  -C (\log m_q)^\gamma\right\}.     \ee
Now, to see that $\CC^{\sss{(r)}}_{T_r({\beta_n})}$ and $\CC^{\sss{(b)}}_{T_b({\beta_n})}$ are disjoint, we apply the following procedure:
It is easy to see that $H(\CC^{\sss{(r)}}_{T_r({\beta_n})-\ell})$ is maximised at $\ell=0$. Hence, we grow the red cluster first until time $T_r({\beta_n})$, and then stop it. Then, we grow the blue cluster step by step, looking at the pairs of half-edges\footnote{Recall that $\CH(A), H(\CA)$ denote the set and number of half-edges attached to vertices in the set $\CA$.} in $\CH(\CC^{\sss{(b)}}_1), \CH(\CC^{\sss{(b)}}_2), \dots, \CH(\CC^{\sss{(b)}}_{T_b({\beta_n})-1})$, and at each step we check whether any of the half edges paired are actually paired to a red half-edge. If this happens for any time before or at $T_b({\beta_n})-1$, then an early connection happens and the distance is at most  $T_b({\beta_n})+T_r({\beta_n})$. (Note that the distance is $T_r({\beta_n})+i$ if we pair a blue half-edge attached to $\CC^{\sss{(b)}}_{i-1}$ to a red half-edge.)

The probability that there is a connection before or at $t(n^{\vr_n'})$ is of the same order of magnitude as the probability that there is a connection at time $t(n^{\vr_n'})$, since the total degree in the whole BP is the same order of magnitude as the total degree in the last generation, thus it is enough to investigate
the probability that $\CH(\CC^{\sss{(b)}}_{T_b({\beta_n})-\ell})$ connects to $\CH(\CC^{\sss{(r)}}_{T_r({\beta_n})})$ for some $\ell\le i_{\star \sss{(b)}}$. This probability is at most
\be\label{eq:conn-ell}   \Pv_{\sss{Y}}\left(\CC^{\sss{(b)}}_{T_b({\beta_n})-\ell} \leftrightarrow \CC^{\sss{(r)}}_{T_r({\beta_n})} \mid  H(\CC^{\sss{(b)}}_{T_b({\beta_n})-\ell}),  H(\CC^{\sss{(r)}}_{T_r({\beta_n})}) \right) \le \frac{H(\CC^{\sss{(b)}}_{T_b({\beta_n})-\ell}) H(\CC^{\sss{(r)}}_{T_r({\beta_n})})}{\ell_n(1+o(1))}.  \ee
Let us write $\CD_n:=\{H(\CC^{\sss{(b)}}_{T_b({\beta_n})-\ell})\le \widehat u_{ i_{\star \sss{(b)} }-\ell},  \forall \ell \in [i_{\star \sss{(b)} }]\}$. Then by \eqref{eq:half-edge-red}, $\CD_n$ happens $\Pv_{\sss{Y}}$-whp. Using this, we sum the bound on the rhs over $\ell\in [i_{\star \sss{(b)}}]$, using \eqref{eq:uiq-asymp}, to obtain that
\be \label{eq:early-error} \Pv_{\sss{Y}}( \CC^{\sss{(r)}}_{T_r({\beta_n})} \cap \CC^{\sss{(b)}}_{T_b({\beta_n})}\neq \varnothing \mid \CD_n ) \le \frac{\widehat u_{ i_{\star \sss{(r)} }}^{\sss{(r)}} }{\ell_n}\sum_{\ell=1}^{ i_{\star \sss{(b)}} }\widehat u_{ i_{\star \sss{(b)} }-\ell}^{\sss{(q)}}   \lesssim  \frac{n^{{\beta_n}(\tau-2)^{\bnr({\beta_n})}}}{\ell_n} \sum_{\ell=1}^{i_{\star \sss{(b)}}}n^{{\beta_n}(\tau-2)^{\bnb({\beta_n})+\ell}}, \ee
where we recall that $\lesssim$ means inequality up to multiplicative factors that are of order at most $\exp\{ (\log n)^\theta\}$ for some $\theta\in [0,1)$, as in Definition \ref{def:sim}.
 Since $\ell_n=\Ev[D_n] n$ is of order $n$, the exponent of $n$ in the dominant term in the numerator is
 \[ {\beta_n}((\tau-2)^{\bnr({\beta_n})}+ (\tau-2)^{\bnb({\beta_n})+1}) < 1, \]
 as long as ${\beta_n}<1/(\tau-1)$, since $\bnb({\beta_n}) , \bnr({\beta_n}) \in [0,1)$. When ${\beta_n}=1/(\tau-1)$, the strict inequality still holds as long as $(\bnb({\beta_n}),\bnr(a{\beta_n}))\neq(0,0)$, an event that happens with probability $1$ under Assumption \ref{assu:pointmass}, since $b_n^{\sss{(q)}}({\beta_n})=0$ is only possible if $Y_q^{\sss{(n)}}$ takes values in a measure $0$ discrete set, see \eqref{def:ti-bi}. The probability of this event tends to $0$ under Assumption \ref{assu:pointmass}.

 For  large enough $n$ the multiplicative factors hidden in the $\sim$ sign on the rhs of \eqref{eq:early-error} are negligible, thus the rhs of \eqref{eq:early-error} tends to zero with $n$.
This finishes the proof of the proposition.
\end{proof}
\section{Distances in the graph}\label{s:main-proof}
\begin{proof}[Proof of Theorem \ref{thm:distances}]
The lower bound is easier, since we can use the first moment method (i.e., Markov's inequality) on the number of paths emanating from $H(\mathcal R_{T_r({\beta_n})})$ and $H(\mathcal B_{T_b({\beta_n})})$ and connecting to each other to obtain a lower bound.
Thus, let us start counting paths of length $z+1$ (that is, $z$ vertices in between) connecting $\CH(\CC^{\sss{(q)}}_{T_q({\beta_n})})$, for $q\in\{r,b\}$. Starting with \eqref{eq:combi}, the restriction now is that $\pi_0 \in \CC^{\sss{(r)}}_{T_r({\beta_n})}$, while $\pi_{z+1} \in \CC^{\sss{(b)}}_{T_b({\beta_n})}$, and there are no restrictions on the in-between vertices. Thus, we obtain a similar formula as in \eqref{eq:A5}, except that the restrictions on the vertices $\pi_i$ are now different:
\be\ba\label{eq:free-from-end} \Ev_{\sss{Y}}&[N_z(\mathcal C^{\sss{(r)}}_{T_r({\beta_n})}, \mathcal C^{\sss{(b)}}_{T_b({\beta_n})})\mid \text{NoBad}] \\ &\le \e^{Cz^2 /\ell_n}\!\!\!\!\!\sum_{\pi_0 \in \CC^{\sss{(r)}}_{T_r({\beta_n})} } \!\!\!\!d_{\pi_0} \cdot\prod_{i=1}^{z }\left(\sum_{\pi_i \in [n]} \frac{d_{\pi_i} (d_{\pi_i}-1)}{\ell_n^\star}\right)\cdot \sum_{\pi_{z+1} \in \CC^{\sss{(b)}}_{T_b({\beta_n})} } \frac{d_{\pi_{z+1}}}{\ell_n^\star}.
\ea \ee
We have to check that in this case $\ell_n^\star=\ell_n(1-o_{\Pv}(1))$ is still satisfied. This follows from the proof of Proposition \ref{lem:no-early-meeting} and the fact that Lemma \ref{lem:badpaths} holds and thus the total number of used half-edges can be bounded from above as the sum of $\sim\huiq$ over $i\le i_{\star\sss{(q)}}$ and $q\in\{r,b\}$.

The first and last sum are handled by \eqref{eq:half-edge-red-bound} in Proposition \ref{lem:no-early-meeting}. We would like to estimate the sums within the product sign on the rhs. For this, recall that $D_n^\star$ stands for the (degree-1) of a vertex that a uniformly chosen half-edge in $\ell_n$ is attached to. $D_n^\star$ then follows distribution $F_n^\star$, see \eqref{def:size-biased1}.
By Claim \ref{cl:tech}, \eqref{eq:second-mom},  the size-biased empirical moment of $D_n^\star$ is
\be\label{eq:edn}\Ev[D_n^\star]=\sum_{\substack{ v\in[n]\\ d_v\le n^{{\beta_n}}}} \frac{d_v (d_v-1)}{\ell_n}\le n^{{\beta_n}(3-\tau)} L_n^{\star, \text{up} }(n^{{\beta_n}})\lesssim n^{{\beta_n}(3-\tau)}. \ee
Applying this inequality on the rhs of \eqref{eq:free-from-end}
yields that 
\be\label{eq:k-free}\Ev_{\sss{Y}}\left[N_z(\mathcal C^{\sss{(r)}}_{T_r({\beta_n})}, \mathcal C^{\sss{(b)}}_{T_b({\beta_n})})\mid \text{NoBad}\right]\lesssim n^{-1}n^{{\beta_n} (\tau-2)^{b_n^{\sss{(r)}}({\beta_n})}} n^{{\beta_n} (\tau-2)^{b_n^{\sss{(r)}}({\beta_n})}} n^{z {\beta_n} (3-\tau)}.\ee
By Markov's inequality, the probability that there is at least one path connecting  $\mathcal C^{\sss{(r)}}_{T_r({\beta_n})}, \mathcal \mathcal C^{\sss{(b)}}_{T_b({\beta_n})}$ with $z+1$ edges can be bounded from above by the expected number of connections, so we obtain the bound
\be\label{eq:kell-connect} \Pv_{\sss{Y}}\left(\mathrm{d}_G(\mathcal C^{\sss{(r)}}_{T_r({\beta_n})}, \mathcal C^{\sss{(b)}}_{T_b({\beta_n})}) \le z+1 \mid \text{NoBad}\right)\lesssim n^{-1+{\beta_n} \left((\tau-2)^{b_n^{\sss{(r)}}({\beta_n})} + (\tau-2)^{b_n^{\sss{(b)}}({\beta_n})}+z(3-\tau)\right)} .\ee
So, the two clusters  are $\Pv_{\sss{Y}}$-whp disjoint as long as this quantity tends to zero. The smallest value of $z\in \N$ when the rhs of \eqref{eq:kell-connect} does \emph{not} tend to zero is
\be\label{eq:zstar-def}\ba z^\star_n&:= \inf\left\{z\in \N: (\tau-2)^{b_n^{\sss{(r)}}({\beta_n})} + (\tau-2)^{b_n^{\sss{(b)}}({\beta_n})}+ z(3-\tau) > 1/{\beta_n} \right \}\\
&\ =\left\lceil    \frac{1/{\beta_n}-(\tau-2)^{b_n^{\sss{(r)}}({\beta_n})}-(\tau-2)^{b_n^{\sss{(b)}}({\beta_n})}}{3-\tau} \right\rceil.  \ea\ee
Since $z_n^\star$ counts the number of \emph{vertices} needed between $\CC^{\sss{(r)}}_{T_r({\beta_n})}$ and $\CC^{\sss{(b)}}_{T_b({\beta_n})}$, and we would like to count the number of \emph{edges}\footnote{The number of edges on a path with $k$ in-between vertices is $k+1$.}, as long as the number of edges between $\CC^{\sss{(r)}}_{T_r({\beta_n})}$ and $\CC^{\sss{(b)}}_{T_b({\beta_n})}$ is at most $z_n^\star$, the bound in \eqref{eq:kell-connect} tends to zero as $n\to \infty$. This in turn means that $\Pv_{\sss{Y}}$-whp there is no path of length $T_r({\beta_n})+T_b({\beta_n})+z_n^\star$ connecting $v_r, v_b$.
Thus, we obtain that $\Pv_{\sss{Y}}$-whp:
\be\label{eq:dist-lower-1} \mathrm d_G(v_r, v_b)\ge T_r({\beta_n})+ T_b({\beta_n})+\left\lceil   \frac{1/{\beta_n} -(\tau-2)^{b_n^{\sss{(r)}}({\beta_n})}-(\tau-2)^{b_n^{\sss{(b)}}({\beta_n})}}{3-\tau} \right\rceil +1. \ee
This completes the proof of the lower bound on $\mathrm{d}_G(v_r, v_b)$ in Theorem \ref{thm:distances}.
For the upper bound on $\mathrm{d}_G(v_r, v_b)$, we expect the existence of a path of length as the rhs of \eqref{eq:dist-lower-1}. To be able to show this, we apply the second moment method. Recall that we have already constructed paths of length $T_q({\beta_n})$ between $v_q$ and a vertex $v_q^\star$, where $d_{v_q^\star}$ as in \eqref{eq:ui*}.
We calculate the expected number and variance of paths of length $z^\star+1$ connecting  $v_r^\star$ to $v_b^\star$, with certain restrictions. Namely, the formula for the variance turns out to be simpler and easier if we count paths where the $i$th vertex on the path falls into a different (and disjoint) set for all $i\ge 0$. The reason why the variance is easier to calculate is that two possible paths can overlap only in fairly simple ways (see \cite[Figure 7]{BarHofKom14}).

Note that since $\bnb, \bnr \in [0,1)$,
\be\label{eq:mbetan} z^\star_n+2 \le \left\lceil \frac{1/{\beta_n}-2(\tau-2)}{3-\tau}\right\rceil+2=:M_{\beta_n}.\ee
Now we divide the set of vertices
into $M_n$ many roughly equal disjoint sets. We denote the $i$th set by $\Delta_i$.  By roughly equal we mean that the following inequalities hold for some $0<c_1< c_2<\infty$
\begin{align}\label{eq:nu-new} \nu_i^{\text{new}}:=\sum_{v\in \Delta_i} \frac{d_i (d_i-1)}{\ell_n} &\in \left[\frac{c_1}{M_{\beta_n}}, \frac{c_2}{M_{\beta_n}}\right] \cdot  \Ev[D_n^\star] \\
\label{eq:kappa-new} \kappa_i^{\text{new}}:=\sum_{v\in \Delta_i} \frac{d_i (d_i-1)(d_i-2)}{\ell_n}& \in \left[\frac{c_1}{M_{\beta_n}}, \frac{c_2}{M_{\beta_n}}\right] \cdot \Ev[D_n^\star(D_n^\star-1)].
\end{align}
This can be done as long as we distribute the vertices in the intervals $[n^{{\beta_n}(1-\ve)}, n^{{\beta_n}}]$ in an approximately uniform way.
We require that $v_r^\star\in \Delta_0$ and $v_b^\star\in \Delta_{M_{\beta_n}}$ in the partitioning.
We will count the paths on vertices $(v_r^\star:=v_0, v_1, \dots, v_{z}, v_{z+1}:=v_b^\star)$ that satisfy the property that the $j$th vertex falls into $\Delta_j$ when $j\le z/2$ and it falls into $\Delta_{M_{\beta_n}+1-j}$ when $j>z/2$.\footnote{This somewhat weird containment is needed since $z+2<M_{\beta_n}$ can also occur.}
As a result of the restriction, the proof of \cite[Lemma 7.1, (7.5)]{BarHofKom14} applies word by word.
This proof bounds the expected number and variance of restricted paths between vertices $a, b$ with $k$ vertices in between.
In this proof, we only need to replace $\nu_i$ by $\nu_i^{\text{new}}$ as in \eqref{eq:nu-new} and $\kappa_i$  by $\kappa_i^{\text{new}}$ as in \eqref{eq:kappa-new}. In our case the degree of vertex $a:=v_r^\star$ is $d_a:=u_{i_{\star \sss{(r)}}}^{\sss{(r)}}$ while the degree of vertex $b:=v_b^\star$ is $d_b:=u_{i_{\star \sss{(b)}}}^{\sss{(b)}}$.

First we need a lower bound on the expected number of paths of length $z+1$ between $v_r^\star, v_b^\star$. We expect it to be of a similar order of magnitude as the upper bound in \eqref{eq:k-free}. The differences between a lower and an upper bound (see \eqref{eq:free-from-end}) are the following:
(1) the first and last factor in \eqref{eq:free-from-end} in lower bound changes, since  we only count the degree of $v_q^\star$,
(2) in the middle factor  in \eqref{eq:free-from-end} we have to apply the restriction that $\pi_j\in\Delta_{j}$ for $j<z/2$ while $\pi_j\in \Delta_{M_{\beta_n}-j}$ for $j> z/2$ instead of summing over all vertices in $[n]$.

Using \eqref{eq:degree-hub}, combined with \eqref{eq:nu-new}
yields that we have the upper and lower bound
\be\label{eq:expected-path-1} \Ev_{\sss{Y}}[N_{z_n^\star}(v_r^\star, v_b^\star)] \ {\buildrel \in \over \sim}\    \frac{n^{{\beta_n} (\tau-2)^{b_n^{\sss{(b)}}({\beta_n})}} n^{{\beta_n} (\tau-2)^{b_n^{\sss{(r)}}({\beta_n})}}}{\ell_n} n^{z {\beta_n} (3-\tau)} \cdot \left[ (c_1/M_{\beta_n})^{z}, (c_2/M_{\beta_n})^{z}\right], \ee
where ${\buildrel \in \over \sim}$ means containment in an interval, where an additional factor of order at most $\exp\{\pm (\log n^{\beta_n})^\theta \}$ for some $\theta<1$, as defined in Definition \ref{def:sim}, might multiply the prefactors of the interval. This additional factor comes from $v_q^\star$ not being precisely equal to $n^{{\beta_n}(\tau-2)^{\bnq({\beta_n})}}$, as well as $\Ev[D_n^\star]$ in \eqref{eq:edn} is not precisely equal to $n^{{\beta_n}(3-\tau)}$.

By the definitions of $M_{\beta_n}$ and $z_n^\star$ in \eqref{eq:mbetan} and \eqref{eq:zstar-def} and the bound $\beta_n\ge (\log n)^{-\gamma}$ for some $\gamma< 1$, for $i\in\{1,2\}$, for all $z\le z_n^\star$,
\[ (c_i/M_{\beta_n})^{z}\ge \exp\left\{ -\gamma \log \log n \cdot (\log n)^\gamma+\log c_i (\log n)^\gamma\right\} \ge \exp\{ -(\log n)^{\theta}\} \] 
for any $\theta\in(\gamma,1)$ and $n$ sufficiently large,
so the lower bound on $\Ev_{\sss{Y}}[N_{z^\star}(v_r^\star, v_b^\star)]$ in \eqref{eq:expected-path-1} fits the Definition \ref{def:sim} when we use that $\ell_n$ is of order $n$. Combining this with the upper bound in \eqref{eq:k-free}, we arrive at the desired
\be\label{eq:expected-path-2} \Ev_{\sss{Y}}[N_{z}(v_r^\star, v_b^\star)] \sim n^{-1} n^{{\beta_n} (\tau-2)^{b_n^{\sss{(b)}}({\beta_n})}} n^{{\beta_n} (\tau-2)^{b_n^{\sss{(r)}}({\beta_n})}} n^{z {\beta_n} (3-\tau)}. \ee

The smallest value of $z$ for which this expression tends to infinity (and not to $0$) as $n\to \infty$ is precisely $z_n^\star$ defined in \eqref{eq:zstar-def}.

By Chebyshev's inequality,
\be\label{eq:cheb} \Pv_{\sss{Y}}(N_{z_n^\star}(v_r^\star, v_b^\star ) = 0) \le \frac{\Var_Y[N_{z_n^\star}(v_r^\star, v_b^\star) ]}{\Ev_{\sss{Y}}[N_{z_n^\star}(v_r^\star, v_b^\star) ]^2}\ee
 and thus, to show the $\Pv_{\sss{Y}}$-whp existence of at least $1$ path of length $z_n^\star+1$, it is enough to show that the variance is of smaller order than the expectation squared.

Thus next we calculate the variance of $N_{z_n^\star}(v_r^\star, v_b^\star)$. Here we rely on the proof of \cite[Lemma 7.1]{BarHofKom14}. Unfortunately here the appendix of the paper does not state the variance independently of the statement of \cite[Lemma 7.1]{BarHofKom14}. However, one can word-by-word follow the derivation of the variance, starting from the formula before \cite[(A.19)]{BarHofKom14} until  \cite[(A.25)]{BarHofKom14}. The first occurence where the setting of this paper deviates from that paper is \cite[(A.26)]{BarHofKom14}: in this equation, in the last error factor $(1+\ve)^k$, $\ve$ can be set to $0$. This can be done since the events in $\CE'_n$ defined in \cite[(A.15)]{BarHofKom14} hold for $\ve=0$ under Assumption \ref{assu:degree-dist}.

A crucial observation in the proof there is the inequality in \cite[(A.27)]{BarHofKom14} stating that
$\kappa_i/\nu_i^2 \le \kappa_1/\nu_1^2$. This allows us to replace every occurrence of $\kappa_i/\nu_i^2$ by $\kappa_1/\nu_1^2$ in that proof. Note that this inequality in our case is not valid, however, with our choice of the $\Delta_i$ it is true that
\[ \frac{\kappa_i^{\text{new}}}{(\nu_i^{\text{new}})^2} \le C \frac{\kappa_1^{\text{new}}}{(\nu_1^{\text{new}})^2}, \]
with $C:=c_2/c_1^2$ from \eqref{eq:nu-new}. Thus when following the proof, we are allowed to replace every occurrence of $\kappa_i/\nu_i^2$ by $C\kappa_1^{\text{new}}/(\nu_1^{\text{new}})^2$. Similarly, $1/\nu_{k-1}$ can be replaced by $\wit C/\nu_{1}^{\text{new}}$.
If we do this replacement in  \cite[(A.28)]{BarHofKom14}, and thus the geometric sums in the formula before \cite[(A.29)]{BarHofKom14} yield that in \cite[(A.29)]{BarHofKom14}
we should replace the two occurrences of $\nu_{k-1}/(\nu_{k-1}-1)$ in \cite[(A.29)]{BarHofKom14} by
\[  \frac{\nu_{1}^{\text{new}}}{\nu_{1}^{\text{new}}-\wit C} \]
 From here on, the arguments work word-by-word again and thus we obtain that the arguments as well as the formulas until \cite[(A.32)]{BarHofKom14} remain all true when implementing these modifications.

Ultimately, the final estimate for the variance is the sum of the rhs of \cite[(A.19),(A.20), (A.29) and (A.32)]{BarHofKom14}  with the addition of the prefactor $C$ at places where
one sees $\kappa_1/\nu_1^2$ and modifying $\nu_{k-1}$ in numerators to $\nu_1$ and $1/(\nu_{k-1}-1)$ to $1/(\nu_{1}-\wit C)$. Using that $d_{v_q^\star}\ge u_{i_{\star \sss{(q)}}}^{\sss{(q)}}$, we thus obtain that (ignoring the `new' superscript everywhere for brevity now):
\be\ba\label{eq:var-est}
&\Var_Y[N_{z_n^\star}(v_r^\star, v_b^\star)] \le \Ev_{\sss{Y}}[N_{z_n^\star}(v_r^\star, v_b^\star)] \\
&\ \ +\overline{\Ev_{\sss{Y}}[ N_{z_n^\star}(v_r^\star, v_b^\star)]}^2 \left(\frac{\nu_1}{\nu_1-\wit C} \frac{C \kappa_1}{\nu_1^2} \Big( \frac{1}{u_{i_{\star \sss{(r)}}}^{\sss{(r)}}}+ \frac{1}{u_{i_{\star \sss{(b)}}}^{\sss{(b)}}}\Big)\!+\!  \frac{\nu_1^2}{(\nu_1-\wit C)^2 }\frac{C^2 \kappa_1^2}{\nu_1^4}\frac{1}{u_{i_{\star \sss{(r)}}}^{\sss{(r)}}u_{i_{\star \sss{(b)}}}^{\sss{(b)}}}\right.\\
&\ \ \left. +\,\frac{8(z_n^\star)^2}{\ell_n} + \left(1+ \frac{C \kappa_1 \nu_1}{\nu_1^2 u_{i_{\star \sss{(r)}}}^{\sss{(r)}} } \right)\left(1+ \frac{C \kappa_1 \nu_1}{\nu_1^2 u_{i_{\star \sss{(b)}}}^{\sss{(b)}} }\right) \frac{z^\star}{\nu_1-C}\left( 2\frac{(z_n^\star)^2 \nu_1}{\ell_n} \frac{C^2 \kappa_1^2}{\nu_1^4} \right) \right) 
\ea\ee
where $\overline{\Ev_{\sss{Y}}[ N_{z_n^\star}(v_r^\star, v_b^\star)]}$ stands for the upper bound on  $\Ev_{\sss{Y}}[N_{z_n^\star}(v_r^\star, v_b^\star)]$ in \eqref{eq:k-free}. Recall that $\overline{\Ev_{\sss{Y}}[ N_{z_n^\star}(v_r^\star, v_b^\star)]}$ and $\Ev_{\sss{Y}}[ N_{z_n^\star}(v_r^\star, v_b^\star)]$ are all given by \eqref{eq:expected-path-2} up to smaller order correction terms.
So, in order to show that the rhs of \eqref{eq:cheb} tends to zero, it is enough to analyse the factor multiplying $\overline{\Ev_{\sss{Y}}[ N_{z_n^\star}(v_r^\star, v_b^\star)]}^2$ in \eqref{eq:var-est} and show that it tends to zero as $n\to\infty$. We carry out this now.

By the same method as the one in Claim \ref{cl:tech}, (i.e., using Karamata's theorem)
it is not hard to show that
\be\label{eq:kappa-order} \Ev[D_n^\star (D_n^\star-1)]=\sum_{\substack{ v\in[n]\\ d_v\le n^{\beta_n}}} \frac{d_v (d_v-1)(d_v-2)}{\ell_n}\le n^{{\beta_n}(4-\tau)} C L_n^{\star, \text{up} }(n^{\beta_n})\lesssim n^{{\beta_n}(4-\tau)}. \ee
This together with \eqref{eq:kappa-new} implies that $\kappa_1\sim n^{{\beta_n}(4-\tau)}$. From  \eqref{eq:edn} and \eqref{eq:nu-new}, $\nu_1\sim n^{{\beta_n}(3-\tau)}$,  and finally from \eqref{eq:ui*}, (see also \eqref{eq:degree-hub}) $u_{i_{\star \sss{(q)}}}^{\sss{(q)}}\sim
n^{{\beta_n}(\tau-2)^{b_n^{\sss{(q)}}({\beta_n}) }}$.
This implies that for $q\in\{r,b\}$
\be\label{eq:first-term-1} \frac{\kappa_1}{\nu_1^2}\frac{1}{u_{i_{\star \sss{(q)}}}^{\sss{(q)}}}\sim n^{{\beta_n}((\tau-2)-(\tau-2)^{b_n^{\sss{(q)}}({\beta_n}) })} \to 0, \ee
as $n\to \infty$, since $\bnb({\beta_n}), \bnr({\beta_n}) <1$, and the error terms hidden in the $\sim$ sign are at most $\exp\{\pm(\log n^{\beta_n})^\theta\}$ for $\theta<1$, and are thus of smaller order of magnitude.
Thus, both the first and second term multiplying $\overline{\Ev_{\sss{Y}}[ N_{z^\star}(v_r^\star, v_b^\star)]}$ in \eqref{eq:var-est} tends to zero as $n\to \infty$. Note that for \eqref{eq:second-term-1} to tend to zero it is crucial that $d_{v_q^\star}>n^{\beta_n(\tau-2)}$, i.e., that $v_q^\star$ is a hub.
When distributing the product in the second line of  \eqref{eq:var-est}, using that $\nu_1/(\nu_1-\wit C)$ is a constant factor, we see that the main contribution comes from the terms
\be\label{eq:second-term-1}\frac{\kappa_1}{\nu_1 u_{i_{\star \sss{(q)}}}^{\sss{(q)}}} \cdot \frac{1}{\ell_n}\frac{\kappa_1^2}{\nu_1^4} \sim n^{{\beta_n}(1-(\tau-2)^{b_n^{\sss{(q)}}({\beta_n}) })}\cdot n^{-1} \cdot n^{2{\beta_n}(\tau-2)} = n^{{\beta_n}(\tau-1+\tau-2-(\tau-2)^{b_n^{\sss{(q)}} ({\beta_n})})-1}. \ee
Note that since $b_n^{\sss{(q)}}({\beta_n})\in[0,1)$, $(\tau-2)-(\tau-2)^{\bnq}\le 0$, and thus the exponent is always negative when ${\beta_n}<1/(\tau-1)$. When ${\beta_n}=1/(\tau-1)$, the exponent is always nonpositive and equals $0$ if and only if $b_n^{\sss{(q)}}({\beta_n})=0$. This is only possible if $Y_q^{\sss{(n)}}$ takes values in a measure $0$ discrete set, see \eqref{def:ti-bi}. The probability of this event tends to $0$ under Assumption \ref{assu:pointmass}.

Combining the estimates in \eqref{eq:first-term-1} and \eqref{eq:second-term-1}, we obtain that the variance of $N_{z_n^\star}(v_r^\star, v_b^\star)$ is of smaller order than its expectation squared, hence the rhs of \eqref{eq:cheb} tends to zero. This establishes that whp there is a path of length $z_n^\star+1$ connecting $v^\star_r$ to $v_b^\star$. Thus, we obtain the existence of a path of length as in \eqref{eq:dist-beta-small-p2}.
This proves the upper bound on $\mathrm{d_G}(v_r, v_b)$ and thus completes the proof of Theorem \ref{thm:distances}.
\end{proof}
\section{Extensions and by-products}\label{s:extension}

In this section  prove Theorems \ref{thm:structure-1} and \ref{thm:structure-2} and sketch the proof of Observation \ref{obs:factor}.

\begin{proof}[Proof of Theorem \ref{thm:structure-1}]
The proof of the first statement of the theorem, that is, \eqref{eq:alpha-truncation} follows from the proof of Theorem \ref{thm:distances} in Section \ref{s:main-proof}.
Recall that we write $\Lambda_{\le z}=\{w\in [n]: d_w \le n^{z}\}$.  Note that in this case, Claim \ref{cl:tech} yields that the empirical second moment restricted to degrees at most $n^{\wit{\beta_n}}$ is $\lesssim n^{\wit{\beta_n}(3-\tau)}$. The proof of Theorem \ref{thm:structure-1} is even simpler than the one in Section \ref{s:main-proof}: here, one does not need to estimate the number of half-edges attached to $v_1, v_2$, since these are given, so there is no need to condition on the sigma algebra $\CF_{\vr_n'}$  or on the good event ($\text{NoBad}$) either. Thus, below in this proof, we can use the `usual' probability measure $\Pv$ instead of $\Pv_Y$.
Let us write $\CP_{v_1, v_2}(z)$ for a path of length $z$ connecting $v_1, v_2$. Counting paths restricted to the set $\Lambda_{\wit{\beta_n}}$, the restriction in the sum in the middle factor in \eqref{eq:free-from-end} is that $\pi_i\in \Lambda_{\le \wit{\beta_n}}$, thus, \eqref{eq:k-free} turns into
 \be\label{eq:k-free-3} \Ev[\#\{\CP_{v_1, v_2}(z) \in \Lambda_{\le\wit{\beta_n}}\} ] \lesssim n^{-1} n^{x_1\wit{\beta_n}+x_2\wit{\beta_n}} n^{z \wit{\beta_n}(3-\tau)} \ee
and \eqref{eq:kell-connect} becomes
\be \Pv\left( \exists \CP_{v_1, v_2}(z) \in \Lambda_{\le \wit{\beta_n}}: |\CP_{v_1, v_2}|\le z+1 \right)\lesssim n^{-1} n^{\wit{\beta_n}(x_1+x_2 + z (3-\tau))}. \ee
 From here, the proof of the upper bound in Section \ref{s:main-proof} can be repeated with ${\beta_n}$ replaced by $\wit{\beta_n}$, $(\tau-2)^{\bnq(\beta_n)}$ replaced by $x_j$, yielding \eqref{eq:alpha-truncation}. The proof of the lower bound in Section \ref{s:main-proof} is again valid word-by-word.
 It is important to note that the second moment method in \eqref{eq:cheb} works only if the variance is of smaller order than the expectation squared, which, in turn, is equivalent to the quantities in \eqref{eq:first-term-1} and \eqref{eq:second-term-1} tending to zero, which is the case whenever $x_j > \tau-2$. (Thus, for non-hub vertices the method does not work.)
 Next we show \eqref{eq:nof-shortest} and \eqref{eq:convergence-nof-paths}.
Note that \eqref{eq:convergence-nof-paths} is equivalent to
\be\label{eq:wrong-interval} \Pv\left( \{\# \CP_{v_1, v_2}^\star \in \Lambda_{\le \wit{\beta_n}}\} \notin [ n^{\wit{\beta_n} f^u(\wit{\beta_n}, x_1, x_2) (1-\ve) }, n^{\wit{\beta_n} f^u(\wit{\beta_n}, x_1, x_2) (1+\ve) }]   \right) \to 0.\ee
We prove this statement using Chebyshev's inequality. Let us shortly write 
\be\label{eq:zetan} \zeta_n:=\{\# \CP_{v_1, v_2}^\star \in \Lambda_{\le \wit{\beta_n}}\}.\ee First, when we write $\underline{\Ev[\zeta_n]}, \overline{\Ev[\zeta_n]}$ for the lower and upper bounds on $\Ev[\zeta_n]$, respectively, then these are handled in  \eqref{eq:expected-path-2}, and equal $\sim n^{\wit{\beta_n} f^u(\wit{\beta_n}, x_1, x_2)}$ by elementary calculations using the value $z_n^\star$ from \eqref{eq:zstar-def} and $f^u$ from \eqref{def:c}.

Next, $\Var[ \zeta_n]$ is handled in \eqref{eq:var-est}, where consecutively in \eqref{eq:first-term-1} and \eqref{eq:second-term-1} it is established that
\be \Var[ \zeta_n] \le C \overline{\Ev[\zeta_n]}^2  \max_{j=1,2}\max\left\{ \frac{\kappa_1}{\nu_1^2 d_{v_j}} ,\frac{\kappa_1^3}{\nu_1^5 \ell_n d_{v_j}} \right\}.\ee
 Comparing the rhs of \eqref{eq:first-term-1} to the rhs of \eqref{eq:second-term-1} with ${\beta_n}$ replaced by $\wit{\beta_n}$,
it is elementary to check that the dominating expression is the rhs of \eqref{eq:first-term-1} unless $\wit{\beta_n}=1/(\tau-1)$, in which case both terms are of the same order.
Nevertheless, we arrive to
\be \max_{j=1,2} \max\left\{ \frac{\kappa_1}{\nu_1^2 d_{v_j}}, \frac{\kappa_1^3}{\nu_1^5 \ell_n d_{v_j}} \right\} \sim  n^{-\wit{\beta_n} (\min(x_1, x_2) - (\tau-2) )}. \ee
Note that this is the point where it becomes clear why the assumption $x_1, x_2 > \tau-2$ was necessary: only in this case can we expect any concentration of the variable $\zeta_n$.
Chebyshev's inequality yields that for any $c(n,\ve)$ that depends on $n$ and some $\ve>0$ to be chosen later
\be\label{eq:cheb-222}\ba   \Pv\left( |\zeta_n - \Ev[\zeta_n]|\ge  \Ev[\zeta_n] c(n,\ve) \right) &\le \frac{\Var( \zeta_n) }{\Ev[\zeta_n]^2 c(n,\ve)^2 }\lesssim  \frac{\overline{\Ev[\zeta_n]}^2 }{\underline{\Ev[\zeta_n]}^2} \frac{1}{c(n,\ve)^2}n^{-\wit{\beta_n} (\min(x_1, x_2) - (\tau-2) )} \\
&\sim\frac{1}{\wit{\beta_n}^2}\frac{1}{c(n,\ve)^2}n^{-\wit{\beta_n} (\min(x_1, x_2) - (\tau-2) )}, \ea\ee
where we have used that $\overline{\Ev[\zeta_n]}/\underline{\Ev[\zeta_n]}$ is $1/\wit{\beta_n}$ times a factor that can be merged in the $\sim$ sign. By the assumption that  ${\beta_n} (\log n)^\gamma\to \infty$,  for some $\gamma<1$, setting \[c(n,\ve):=n^{-(1-\ve)\wit{\beta_n} (\min(x_1, x_2) - (\tau-2) )/2}\to 0,\] the rhs tends to zero for any fix $\ve>0$. Note that we only have upper and lower bounds on the expected value, thus we obtain that
\be\label{eq:concentration-nof-paths} \Pv(\zeta_n \notin [(1-c(n,\ve))\underline{\Ev[\zeta_n]} , (1+c(n,\ve))\overline{\Ev[\zeta_n]}] )\to 0.\ee
The interval in \eqref{eq:wrong-interval} is certainly wider than the one excluded here. This finishes the proof of \eqref{eq:convergence-nof-paths}.   Further note that both the lower as well as the upper ends of the interval in \eqref{eq:concentration-nof-paths} are still $\sim n^{\wit{\beta_n} f^u(\wit{\beta_n}, x_1, x_2)}$. This finishes the proof of \eqref{eq:nof-shortest}.

We turn to the proof of \eqref{eq:no-low-degree}, which is essentially Markov's inequality.
 Indeed, when we require for some $\delta<1$ that one of the vertices on a path must fall in $\Lambda_{\le\delta\wit{\beta_n}}$ (but not its location), we obtain that the expected number of  such paths connecting $v_1, v_2$ is at most
\be\label{eq:z-modified}\Ev\left[ \#\{\CP_{v_1, v_2}(z): \exists i\le z, \pi_i \in \Lambda_{\le\delta\wit{\beta_n}}\} \right]\lesssim z n^{-1} n^{\wit{\beta_n}(x_1+x_2 + (z-1)(3-\tau) + \delta (3-\tau))}.\ee
 Let us denote by $z_n^\star(\delta)$ the smallest $z$ for which this quantity does not tend  to zero, which is exactly
 \[ z_n^\star(\delta)=  \left\lceil    \frac{1/\wit{\beta_n}-x_1-x_2}{3-\tau}+(1-\delta) \right\rceil\]
 On the other hand, without any restriction, the shortest path uses
 \[ z_n^\star(1)=\left\lceil    \frac{1/\wit{\beta_n}-x_1-x_2}{3-\tau} \right\rceil\]
 many in-between vertices, by \eqref{eq:alpha-truncation}. As long as $\delta$ is such that $z^\star(\delta)> z^\star(1)$, the rhs of \eqref{eq:z-modified} tends to $0$ for $z=z^\star(1)$. Thus, there will be no connecting paths of length $z^\star(1)$ that have a vertex in $\Lambda_{\le \delta\wit{\beta_n}}$.
 Finally, the largest $\delta$ that we can achieve, $\delta_{\max}:=\sup\{ \delta: z^\star(\delta)>z^\star(1)\}$ is precisely the lower fractional part of the expression within the upper-integer-part in $z^\star(1)$, in other words, $\delta_{\max}=f^\ell(\wit{\beta_n}, x_1, x_2)$,
 establishing \eqref{eq:no-low-degree}.
\end{proof}
\begin{proof}[Sketch proof of Observation \ref{obs:factor}]
From the proof of Theorem \ref{thm:structure-1} it follows that all shortest paths use degree at least $\sim n^{\wit \beta_n f^\ell(\wit \beta_n,x_1,x_2)}$ while the Chebyshev's inequality in \eqref{eq:cheb-222} shows that the  number of shortest paths is $\sim n^{\wit \beta_n f^u(\wit\beta_n, x_1,x_2)}$. Multiplying these two together yields one part of the observation. The full statement is finished when we show that there is at least one shortest path that actually uses a vertex with degree $\sim n^{\wit \beta_n f^\ell(\wit \beta_n,x_1,x_2)(1+\ve)}$ for arbitrary small $\ve>0$. This can be done using Chebyshev's inequality again in the same way as for $\zeta_n$ in \eqref{eq:zetan}, now counting paths with at least one vertex in $\Lambda_{\le \wit\beta_nf^\ell(\wit\beta_n, x_1,x_2)(1+\ve)}$ and the rest of the vertices in $\Lambda_{\le \wit\beta_n}$. In this case, the variance vs expectation squared method carries through the same way.
\end{proof}
\begin{proof}[Proof of Theorem \ref{thm:structure-2}]
The proof of this theorem is essentially the upper bound -- the construction of the connecting path -- of the proof of Theorem \ref{thm:distances}.
For the first statement, about the segments with increasing degrees, Proposition \ref{prop:hubs} shows that there is a path of length $T_q({\beta_n})$ that connects $v_q$ to a vertex $v_q^\star$ with degree as in \eqref{eq:degree-at-segment}, while Proposition \ref{lem:no-early-meeting} ensures that the total degree in the exploration clusters of depth $T_q({\beta_n})$ is the same order of magnitude as the degree of $v_q^\star$, and that the two exploration clusters are disjoint.

By Lemma \ref{lem:badpaths}, the $(T_q({\beta_n})-k)$th vertex on this path has degree at least $u_{i_{\star \sss{(q)}}-k}^{\sss{(q)}}$ and at most $\widehat u_{i_{\star \sss{(q)}}-k}^{\sss{(q)}}$, both of them $\sim n^{{\beta_n}(\tau-2)^{\bnq({\beta_n})+k }}$.

Thus, when connecting  the two clusters $\CC^{\sss{(r)}}_{T_r({\beta_n})}$ and $\CC^{\sss{(b)}}_{T_b({\beta_n})}$, Theorem \ref{thm:structure-1} can be applied to show that there is a path that connects $v_r^\star, v_b^\star$ of length as in \eqref{eq:connecting-segment}, using vertices of degree at least $n^{{\beta_n}f_n^\ell}$. By the bound on the total degrees in Proposition \ref{lem:no-early-meeting}, we can identify all the vertices in $\CC^{\sss{(q)}}_{T_q({\beta_n})}$, and apply Theorem \ref{thm:structure-1} once more to see that
no shorter path is possible between $\CC^{\sss{(r)}}_{T_r({\beta_n})}$ and $\CC^{\sss{(b)}}_{T_b({\beta_n})}$ than the one with length \eqref{eq:connecting-segment}, and any of these shortest paths uses vertices with degrees at least $n^{{\beta_n}f_n^\ell}$.
 \end{proof}

\section{Comparison of typical distances}\label{s:compare}
In this section we compare the result of Theorem \ref{thm:distances} applied to the ${\beta_n}=1/(\tau-1)$ setting to \cite[Theorem 1.2]{HHZ07} in more detail.
Here we argue that the two formulations - namely the one in \cite[Theorem 1.2]{HHZ07} and the one in \eqref{eq:distance-iid} - are indeed the same, by describing the core idea of the proof of \cite[Theorem 1.2]{HHZ07}, and relate  quantities (events, random variables, etc.) appearing in that proof to quantities in this paper. In this section $\approx$ means equality up to a $(1+o_{\Pv}(1))$ factor.

The proof of \cite[Theorem 1.2]{HHZ07} goes through a minimisation problem, where the two BFS clusters of $v_r$ and $v_b$ should connect the first time such that a coupling (to two branching processes) should be maintained. More precisely, suppose $k_1$ is a random variable that is measurable w.r.t.\ $\{\CC_{s}^{\sss{(q)}}\}_{s=1}^m$, for some $m$ (in this paper we take $m=t(n^{\vr_n'})$).  Suppose we run the BFS started from $v_r$ for $k_1$ steps, and from $v_b$ for $k-k_1-1$ steps. There are $Z_{k_1+1}^{\sss{(r)}}, Z_{k-k_1}^{\sss{(b)}}$ many half edges attached to the vertices in the two clusters, respectively. The distance between $v_r, v_b$ is then larger than $k$ if these sets of half-edges do not connect to each other, and the probability of this event is approximately
\be \label{eq:distance-tail} \Pv(H(\CC^{\sss{(r)}}_{k_1})\cap H(\CC^{\sss{(b)}}_{k-k_1-1})=\varnothing\mid Z_{k_1+1}^{\sss{(r)}}, Z_{k-k_1}^{\sss{(b)}}) \approx \exp\{ - c Z_{k_1+1}^{\sss{(r)}} Z_{k-k_1}^{\sss{(b)}} / \ell_n\} \ee
  A branching process approximation similar to the one in Section \ref{s:couple} is performed to approximate the numerator in the exponent. However, this BP approximation is only valid until none of the colors have more half-edges than $n^{(1-\ve) / (\tau-1)}$ for some small $\ve>0$, i.e., they do not reach the highest-degree vertices in the graph yet. This criterion is established in \cite[Proposition 3.2]{HHZ07}. The set $\CT_{m}^{i,n}$ in \cite[(3.3)]{HHZ07} exactly describes those values of $\ell$ for which $\{\ell\le T_q+1\}$, $q\in\{r,b\}$ holds (where $T_q=T_q(1/(\tau-1))$, defined in \eqref{def:ti-bi}, or in \eqref{eq:k*+i*}, is the time to reach the hubs). The $+1$ is added to $T_q$ since the half-edges attached to vertices in the BFS cluster at time $T_q$ can be described as the next generation, they have size $Z_{T_q+1}^{\sss{(q)}}$.

Now, from \eqref{eq:distance-tail}, we see that $\{\mathrm d_G(v_r, v_b)>k\}$ happens whp if $Z_{k_1+1}^{\sss{(r)}}\cdot Z_{k-k_1}^{\sss{(b)}} = o(n)$ and also that both  $k_1+1 \in \CT_{m}^{r,n}$ and $k-k_1 \in \CT_{m}^{b,n}$ holds. These latter  conditions are described in the set of indices $\CB_n$ in \cite[(4.57)]{HHZ07}:
 \[  \CB_n:=\{ k_1\in \N: k_1+1 \le T_r+1, k-k_1 \le T_b+1\}.\]

Using the BP approximation similar as in \eqref{def:yrn-ybn},  and the rhs of \eqref{eq:distance-tail},
\[ \Pv(\mathrm{d}_G(v_r, v_b) >k) \approx \max_{k_1 \in \CB_n}\exp \{ - C \exp\{(\tau-2)^{-(k_1+1)}\Yrn + (\tau-2)^{-(k-k_1)} \Ybn - \log n  \}   \}. \]
With the event \[ \CE_{n,k}(\delta):=\{ \exists k_1\in \CB_n, (\tau-2)^{-(k_1+1)}\Yrn + (\tau-2)^{-(k-k_1)} \Ybn< (1-\delta)\log n\} \] it is obvious from \eqref{eq:distance-tail} that for any $\delta>0$, $\lim_{n\to \infty}\Pv(\mathrm d_G(v_r, v_b)>k\mid \CE_{n,k}(\delta)) \to 1$, while $\Pv(\mathrm d_G(v_r, v_b)>k\mid \CE_{n,k}^c(\delta)) \to 0$. Hence, we get that
\be\label{eq:distance-tail-BP} \Pv_Y(\mathrm d_G(v_r, v_b)>k) \approx \Pv( \min_{k_1\in \CB_n}(\tau-2)^{-(k_1+1)}\Yrn + (\tau-2)^{-(k-k_1)} \Ybn < \log n  ).\ee
The paper shows that $\min_{k_1\in \CB_n}$ can be replaced by $\min_{k_1\le k}$ in the minimum in \eqref{eq:distance-tail-BP} above.

Next we show that the formulation of \eqref{eq:distance-tail-BP} gives the same distances as our statement for typical distances in Theorem \ref{thm:distances}, that is,
\be\label{eq:our-statement} \mathrm d_G(v_r, v_b) = T_r + T_b + 2 - \ind\{(\tau-2)^{\bnb} + (\tau-2)^{\bnr} > \tau-1\}.\ee
To be able to show that the two formulation are the same, we use \eqref{eq:distance-tail-BP} to show that 
\begin{enumerate}
\item $\Pv( \mathrm d_G(v_r, v_b) > T_r + T_b )\to1$, 
\item $\Pv( \mathrm d_G(v_r, v_b) > T_r + T_b +1 )\to \Pv(\ind\{(\tau-2)^{\bnb} + (\tau-2)^{\bnr} > \tau-1\}=0 )$ and finally that 
\item $\Pv( \mathrm d_G(v_r, v_b) > T_r + T_b +2 ) \to 0$.
\end{enumerate}
From \eqref{eq:k*+i*} and for ${\beta_n}=1/(\tau-1)$, it is an elementary calculation to check that that for any $i\in \Z$, $q\in\{r,b\}$
\be \label{eq:degree-at-tji} \left(\tau-2\right)^{-(T_q+1+i)} Y_q^{\sss{(n)}} = \log n \frac{(\tau-2)^{\bnr-i}}{\tau-1}.  \ee
First, we check that \eqref{eq:distance-tail-BP} gives (a).
For this we set $k=T_r+T_b$, and let us write $k_1:=T_r-\ell$ for some $\ell\in \Z$, then $k-k_1=T_b+\ell$. Hence, we can rewrite \eqref{eq:distance-tail-BP} using  \eqref{eq:degree-at-tji} with $i=-\ell$ and $i=\ell-1$, and get
\[ \Pv_Y( \mathrm d_G(v_r, v_b) > T_r+T_b) \approx  \Pv\left( \min_{\ell} \frac{(\tau-2)^{\bnr+\ell} + (\tau-2)^{\bnb+1-\ell}}{\tau-1} < 1  \right). \]
It is  clear now that setting $\ell=0$ shows that the inequality is satisfied for all $\bnr, \bnb \in [0,1)$, since the expression after the $\min_\ell$, for $\ell=0$ is at most $(1+(\tau-2))/(\tau-1) = 1$.  Moreover, note that for $\ell=0$, $k_1=T_r$ and $k-k_r=T_b$, so both $k_1+1\le T_r+1$ and $k-k_1 \le T_b+1$ hold, hence, we found an index $k_1$ in $\CB_n$.

Next, we check that \eqref{eq:distance-tail-BP} gives (b).
For this, we set $k=T_r+T_b+1$, again write $k_1=T_r-\ell$, so that $k-k_1=T_b+1+\ell$. Hence, we can rewrite \eqref{eq:distance-tail-BP} using \eqref{eq:degree-at-tji} with $i=-\ell$ and $i=\ell$ and get
 \[ \Pv_Y( \mathrm d_G (v_r, v_b)> T_r+T_b+1) \approx  \Pv\left( \min_{\ell} \frac{(\tau-2)^{\bnr+\ell} + (\tau-2)^{\bnb-\ell}}{\tau-1} < 1  \right). \]
It is clear now that setting $\ell=0$ yields (b).  Moreover, note that for $\ell=0$ $k_1\in \CB_n$ holds as well. We argue that $\ell=0$ is indeed the minimizer of the expression after the $\min$. Wlog we can assume that $\ell\ge 1$, the case when $\ell\le -1$ can be treated similarly. Then,
we need to show that for all $\ell\ge 1$,
\be\label{eq:needed}  (\tau-2)^{\bnr+\ell} +  (\tau-2)^{\bnb-\ell} > (\tau-2)^{\bnr} +  (\tau-2)^{\bnb}. \ee
Rearranging this inequality yields
\[   (\tau-2)^{\bnb} ((\tau-2)^{-\ell}-1) > (\tau-2)^{\bnr}(1-(\tau-2)^\ell) \]
Since for all $\ell>1$,
\[ \ba  (\tau-2)^{\bnb} \left( (\tau-2)^{-\ell }-1\right) &> (\tau-2)^{-\ell+1} -(\tau-2) \ge 1-(\tau-2)^{\ell} \ge 
 (\tau-2)^{\bnr} \left( 1-(\tau-2)^{\ell }\right), \ea\]
the claim is established.

Finally, we check (c). For this, we set $k=T_r+T_b+2$, again write $k_1=T_r-\ell$, so that $k-k_1=T_b+2+\ell$. Hence, we can rewrite \eqref{eq:distance-tail-BP} using \eqref{eq:degree-at-tji} with $i=-\ell$ and $i=\ell+1$ and get
 \be\label{eq:(c)} \Pv_Y( \mathrm d_G(v_r, v_b) > T_r+T_b+2) \approx  \Pv\left( \min_{\ell} \frac{(\tau-2)^{\bnr+\ell} + (\tau-2)^{\bnb-\ell-1}}{\tau-1} < 1  \right). \ee
We need to show that no $\ell\in \Z$ satisfies this minimisation problem and thus the probability tends to $0$.
For this, we use again that $\bnr,\bnb \in [0,1)$ implies that
\[(\tau-2)^{\bnr+\ell} + (\tau-2)^{\bnb-\ell-1} > (\tau-2)^{\ell+1} + (\tau-2)^{-\ell},\]
and it is elementary to show again that the rhs is at least $\tau-1$ for all $\ell \in \Z$ and $\tau \in (2,3)$, thus, the inequality on the rhs of \eqref{eq:(c)} cannot be satisfied.

These calculations show that the statement of Theorem \ref{thm:distances} yields - through a non-trivial rewrite - the statement of \cite[Theorem 1.2]{HHZ07}. The final formula of \cite[Theorem 1.2]{HHZ07}, i.e., the distribution of the fluctuation of the typical distance around $2\log\log n/|\log (\tau-2)|$ is then obtained by solving analytically the minimisation problem  on the rhs of \eqref{eq:distance-tail-BP} with $k_1\in \CC^{\sss{(b)}}_n$ replaced by $k_1 \le k$.
\section{Acknowledgement}
This work is supported by the Netherlands
Organisation for Scientific Research (NWO) through VICI grant 639.033.806 (RvdH), VENI grant 639.031.447 (JK), the Gravitation {\sc Networks} grant 024.002.003 (RvdH) and the STAR Cluster (JK).

\bibliographystyle{abbrv}
\bibliography{refscompetition}

\end{document}